\documentclass[oneside,english]{amsart}
\usepackage[T1]{fontenc}
\usepackage[latin9]{inputenc}
\usepackage{color}
\usepackage{babel}
\usepackage{amsthm}
\usepackage{amssymb}
\usepackage{cancel}
\usepackage{graphicx}
\usepackage[unicode=true,
 bookmarks=false,
 breaklinks=false,pdfborder={0 0 1},backref=section,colorlinks=false]
 {hyperref}

\makeatletter
\numberwithin{equation}{section}
\numberwithin{figure}{section}

\usepackage{amsfonts}
\usepackage{amsfonts}\usepackage{amsthm}\usepackage{enumitem}\usepackage{bm}\usepackage{thmtools}
\graphicspath{ {./images/} }

\def\theenumi{\arabic{enumi}}

\def\theenumii{\alph{enumii}}
\def\p@enumii{\theenumi.}

\def\theenumiii{\arabic{enumiii}}
\def\p@enumiii{(\theenumi)(\theenumii)}

\def\p@enumiv{\p@enumiii.\theenumiii}

\newtheorem{theorem}{Theorem}[section]
\newtheorem{corollary}[theorem]{Corollary}\newtheorem{lemma}[theorem]{Lemma}\newtheorem{question}[theorem]{Question}\newtheorem{proposition}[theorem]{Proposition}

\theoremstyle{definition}
\newtheorem{remark}[theorem]{Remark}\newtheorem{example}[theorem]{Example}

\usepackage{babel}

\makeatother

\begin{document}
\title{Law of large numbers for the drift of two-dimensional wreath product}
\author{Anna Erschler and Tianyi Zheng}
\begin{abstract}
We prove the law of large numbers for the drift of random walks on
the two-dimensional lamplighter group, under the assumption that the
random walk has finite $(2+\epsilon)$-moment. This result is in contrast
with classical examples of abelian groups, where the displacement
after $n$ steps, normalised by its mean, does not concentrate, and
the limiting distribution of the normalised $n$-step displacement
admits a density whose support is $[0,\infty)$. We study further
examples of groups, some with random walks satisfying 
LLN for drift and other examples where such concentration phenomenon does not hold, and study relation of this property with asymptotic
geometry of groups. 
\end{abstract}

\thanks{This project has received funding from the European Research Council (ERC) under
the European Union's Horizon 2020 
research and innovation program (grant agreement
No.725773)
The
first named author also thanks the support of the ANR grant MALIN}
\date{November 11 2020}

\maketitle

\section{Introduction}

In this paper we study the limiting behavior of the distance to the
origin of random walks on non-abelian solvable groups, in particular
show a law of large numbers type result for random walks of sublinear
drift on wreath products over $\mathbb{Z}^{2}$ with finite or infinite
lamp groups. This is in contrast with classical central limit theorems
for random walks on $\mathbb{Z}^{d}$, see e.g., \cite[Chapter VIII.4]{Feller}.

Let $G$ be a finitely generated group and $S$ be a finite generating
set of $G$. Denote by $l_{S}$ the word length function on $G$ with
respect to the generating set $S$. Let $\mu$ be a probability measure
on $G$; consider a random walk $\left(X_{n}\right)_{n=0}^{\infty}$
on $G$ with step distribution $\mu$. The \emph{drift function} $L_{\mu}(n)$
is defined as the mean of $l_{S}(X_{n}),$ that is, 
\[
L_{\mu}(n)=\mathbb{E}\left[l_{S}(X_{n})\right].
\]
If $L_{\mu}(n)$ grows linearly, then it is well-known that by Kingman's
subadditive ergodic theorem $l_{S}(X_{n})/n$ tends to a positive
constant almost surely. If $L_{\mu}(n)$ is sublinear, then in general
the sequence of normalized random variables $l_{S}(X_{n})/L_{\mu}(n)$
does not necessarily converge in distribution. When $L_{\mu}(n)\simeq\sqrt{n},$
we say the $\mu$-random walk is diffusive. In contrast to abelian
groups, where symmetric simple random walks have diffusive behavior,
there is a rich spectrum of behavior of $L_{\mu}(n)$ for symmetric
random walks on non-abelian groups: for any subadditive function $f$
between $\sqrt{n}$ and $n$ satisfying certain regularity conditions,
there exists a group $G$ and a symmetric measure $\mu$ on $G$ with
finite generating support such that $L_{\mu}(n)$ is equivalent to
$f$, see \cite{amirvirag,brieusselzheng}.

In many examples, the limiting distribution of $l_{S}(X_{n})/L_{\mu}(n)$
exists and admits a density on $\mathbb{R}_{+}$. For example, this
is the case when $G$ is abelian and $\mu$ is a centered measure
with finite generating support. This fact can be deduced from the
local central limit theorem for random walks on $\mathbb{Z}^{d}$,
see e.g. \cite[Chapter 2.1]{LawlerLimic}.

When $l_{S}(X_{n})/L_{\mu}(n)$ converges to a constant $c>0$ almost
surely, we say that the distance to the origin $l_{S}(X_{n})$ of
the random walk obeys a law of large numbers. To our knowledge,  it has not been studied  up to now when such  almost sure convergence
of $l_{S}(X_{n})/L_{\mu}(n)$ to a constant occurs for random walks
with sublinear drift function $L_{\mu}(n)$. In this paper, we show
a law of large numbers for the displacement of a centered random walk
on the wreath product $\mathbb{\mathbb{Z}}^{2}\wr(\mathbb{Z}/2\mathbb{Z})$.
Recall that the wreath product $H\wr L$ of groups $H$ and $L$ is
a semi-direct product $\left(\oplus_{H}L\right)\rtimes H$, where
$H$ acts by translation on $\oplus_{H}L$. The group $H\wr(\mathbb{Z}/2\mathbb{Z})$
is sometimes called the lamplighter group over $H$ (with the lamp
group $\mathbb{Z}/2\mathbb{Z}$). We will recall this well-known interpretation
as lamplighter in "plan of the proof" section at the end of this
introduction.

The study of random walks on wreath products is initiated by Kaimanovich
and Vershik, \cite{kaimanovichvershik}, who have shown that wreath
products illustrate several phenomena about Poisson boundary on random
walks on groups. One and two dimensional wreath products $\mathbb{Z}\wr\mathbb{Z}/2\mathbb{Z}$
and $\mathbb{Z}^{2}\wr\mathbb{Z}/2\mathbb{Z}$ are examples of groups
of exponential growth where the Poisson boundary of simple random
walks is trivial; $\mathbb{Z}^{d}\wr\mathbb{Z}/2\mathbb{Z}$, $d\ge3$,
are amenable groups with non-trivial boundary of simple random walks;
$G=\mathbb{Z}\wr\mathbb{Z}/2\mathbb{Z}$ admits (infinite entropy)
measures $\mu$ such that the boundary of $(G,\mu)$ is non-trivial
and the boundary defined by an inverse measure $(G,\hat{\mu})$ is
trivial. The study of Poisson boundary on wreath products was continued
in a series of works, among which we mention a recent result of Lyons
and Peres, \cite{lyonsperes}, which provide a complete description
of the Poisson boundary of simple random walks on $\mathbb{Z}^{d}\wr\mathbb{Z}/2\mathbb{Z}$,
$d\ge3$. Many other questions about random walks on wreath products
were studied. Here is a non-exhaustive list of problems: return probability
on wreath products was studied in papers of Varopoulos \cite{varopoulos},
Pittet, Saloff-Coste \cite{pittetsaloffcoste} and Revelle \cite{revelle1},
Law of iterated logarithm was studied in \cite{revelle2},
positive harmonic functions and Martin boundary by Brofferio, Woess
\cite{brofferio,brofferiowoess}, minimal growth of (not necessary
bounded and not necessary positive) harmonic function by Benjamini et al \cite{benjaminietal}, recurrent subsets
(and instability of recurrence of subsets and Green kernel) by Benjamini
and Revelle \cite{bejnaminirevelle}. Our main interest here are infinite
wreath products, but we mention that random walks on finite wreath
products also provide many interesting examples for the study of random
walks on finite graphs, see for instance H\"{a}ggstr\"{o}m and Jonasson \cite{haggsromjonasson},
Peres, Revelle \cite{peresrevelle} and Komj\'{a}thy, Peres \cite{komjathyperes}
for the study of mixing and relaxation time for random walks on these
groups; and Miller, Sousi \cite{millersousi}, Dembo et al \cite{dembodingmillerperes}
for the study of late points of random walks. In papers above and
in our work, we study Markov kernels invariant by group actions. 
The work of Lyons, Pemantle and Peres \cite{lyonspemantle} about \textquotedbl homesick\textquotedbl{}
random walks on these groups do not satisfy this condition: a somehow
counterintuitive example when inward based random walks move quicker
from the origin than simple random walks. Finally, we always discuss
here wreath products for the action of the group on itself and we
do not mention several recent works about random walks on permutational
extensions.

We say that a probability measure $\mu$ on a group $G$ is \emph{centered}
if $\mu$ has finite first moment and $\sum_{g\in G}\chi(g)\mu(g)=0$
for any homomorphism $\chi:G\to(\mathbb{R},+)$. We say that a measure
$\mu$ on $G$ is non-degenerate if its support generates $G$ as
a semigroup. The measure $\mu$ has \emph{finite $\alpha$-moment}
if $\sum_{g\in G}l_{S}(g)^{\alpha}\mu(g)<\infty$. It is clear that
this condition does not depend on the choice of the generating set
$S$. Now we formulate our main result for the displacement of random
walk on two-dimensional lamplighter groups.

\begin{theorem}\label{main}

Let $G=\mathbb{\mathbb{Z}}^{2}\wr(\mathbb{Z}/2\mathbb{Z})$ and $\mu$
be a centered non-degenerate probability measure of finite $(2+\epsilon)$-moment
on $G$, $\epsilon>0$. Let $S$ be a finite generating set of $G$.
Then the $\mu$-random walk $(X_{n})_{n=0}^{\infty}$ on $G$ satisfies
\[
\lim_{n\to\infty}\frac{l_{S}(X_{n})}{n/\log n}=c\mbox{ a.s.}
\]
for some positive constant $c$.

\end{theorem}
We remark that without the second moment condition the normalization can be different from $n/\log n$.
For instance, there are symmetric random walks with finite first moment on $\mathbb{\mathbb{Z}}^{2}\wr(\mathbb{Z}/2\mathbb{Z})$
which have linear speed. 

We also prove that the law of large numbers for random walk displacement
holds in some examples with infinite lamp groups, for instance on
the wreath product $\mathbb{Z}^{2}\wr\mathbb{Z}$. In that case the
argument relies on a result of \v{C}ern\'{y} \cite{cerny}, see Proposition
\ref{2-dim Z-lamp}. Our main goal is to treat the case where the
lamp group is finite as in Theorem \ref{main}.

Now we comment on the situation in contrast with the phenomenon described
in Theorem \ref{main}. By Gromov's polynomial growth theorem \cite{GromovPoly},
groups of polynomial growth are virtually nilpotent. In Alexopoulos
\cite{alexopoulos}, a local central limit theorem is established
for centered finite range random walks on groups of polynomial growth.
As a consequence, for centered finite range random walks on polynomial
growth groups, $l_{S}(X_{n})/\sqrt{n}$ converges in distribution
and the limiting distribution admits a positive density on the ray
$(0,\infty).$ Beyond groups of polynomial growth, such convergence
in distribution can also be observed in some examples of groups of
exponential growth. In Proposition \ref{0-1 lamp} we compute the
limiting distribution of $l_{S}(X_{n})/\sqrt{n}$ for symmetric random
walks of finite second moment on $\mathbb{Z}\wr(\mathbb{Z}/2\mathbb{Z})$.
Note that in the examples mentioned above, the random walks exhibit
diffusive behavior (for an example with non-diffusive drift function
where the limiting distribution of $l_{S}(X_{n})/L_{\mu}(n)$ can
be computed see Example \ref{1-dim infinite lamp}), it is natural
to ask in general the following question.

\begin{question}\label{question}

If the random walk on $G$ is diffusive, that is $L_{\mu}(n)\simeq\sqrt{n},$
is it true that $l_{S}(X_{n})/L_{\mu}(n)$ converges in distribution
to a limiting law whose density charges the whole ray $(0,\infty)$?

\end{question}

It is known that the rate of escape of random walks on infinite groups
always satisfy a diffusive lower bound, see \cite{LeePeres}. In Section
\ref{sec:questions} we formulate and discuss this question under
the stronger assumption that the group $G$ admits what is called
controlled F{\o}lner pairs (see Question \ref{controlled-density}).
We provide some evidence supporting a positive answer. As mentioned
earlier, Proposition \ref{0-1 lamp} describes the limiting distribution
of $l_{S}(X_{n})/L_{\mu}(n)$ for simple random walks on the lamplighter
$\mathbb{Z}\wr F$ over $\mathbb{Z}$ with finite lamp group $F$.
As another evidence, in Lemma \ref{cautious-density}, we show that
if $G$ admits a sequence of controlled F{\o}lner pairs and limiting
density of $l_{S}(X_{n})/L_{\mu}(n)$ exists, then the support of
the limiting density must be the whole ray $(0,\infty).$ The definition
of controlled F{\o}lner pairs is recalled in Section \ref{sec:questions}.
If instead of a sequence of controlled F{\o}lner pairs, we assume
a weakened condition that $G$ admits controlled F{\o}lner pairs
on some scales, then the situation is more complicated. We illustrate this by 
examples of lacunary hyperbolic groups which admit controlled F{\o}lner
pairs on some scales, where the limiting behavior of $l_{S}(X_{n_{i}})/L_{\mu}(n_{i})$
depends on the choice of $(n_{i})$, see Proposition \ref{lacunary}.

We mention that in certain classes of groups where $L_{\mu}(n)$
grows linearly, central limit theorems for $l_{S}(X_{n})-L_{\mu}(n)$
are established: see for example the central limit theorem by Benoist
and Quint on hyperbolic groups \cite{BenoistQuint2016}; and more
generally on acylindrically hyperbolic groups by a different approach
in Mathieu and Sisto \cite{MathieuSisto}. Prior to these works, central
limit theorem for the drift of random walk in the Green metric on
a hyperbolic group is established by Bjorklund in \cite{Bjorklund},
and for quasimorphisms along random walk trajectories in Bjorklund
and Hartnick \cite{BjorklundHartnick}. In Question \ref{question}
one might also ask (similar to the situation of abelian groups where
the classical central limit theorem holds), whether the limiting density
has Gaussian decay at infinity. See more on this in Section \ref{sec:questions}
and in particular a list of questions after Question \ref{controlled-density}.

\subsection*{Plan of the proof of Theorem \ref{main}}

For simple random walk on $\mathbb{Z}^{2}$, a classical result of
Dvoretzky and Erd{\H{o}}s \cite{dvoretzkyerdos} shows that the size
of the range satisfies the strong law of large numbers: 
\begin{equation}
\lim_{n\to\infty}\frac{R_{n}}{\pi n/\log n}=1\mbox{ a.s.,}\label{eq:erdosdovret}
\end{equation}
where $R_{n}$ is the number of vertices visited by the random walk
up to time $n$. This result is generalized in Jain and Pruitt \cite{jainpruitt}
where it is shown that for any recurrent random walk on $\mathbb{Z}^{2}$,
the strong law of large numbers $\lim_{n\to\infty}R_{n}/\mathbb{E}\left[R_{n}\right]=1$
holds almost surely.

We recall that the name lamplighter for wreath product $H\wr\mathbb{Z}/2\mathbb{Z}$
with the two elements group $\mathbb{Z}/2\mathbb{Z}$ comes from the
following observation. The elements of the wreath product are of the
form $(h,f)$, where $h\in H$ and $f:H\to\mathbb{Z}/2\mathbb{Z}$.
One can view the function $f$ as a configuration of \textquotedbl lamps\textquotedbl ,
saying that in each element $h$ of a base group $H$ there is a lamp,
which is on if the value $f(h)$ is equal to $1$ and off if this
value is zero, see Figure \ref{lamplighter3} This interpretation is in particular useful when we
discuss simple random walks on $H\wr\mathbb{Z}/2\mathbb{Z}$: the
random walker walks on $H$, and at the point he visits he can lit
or extinguish the lamp, with some probability specified by the step
distribution.

\begin{figure}
\includegraphics[scale=0.9]{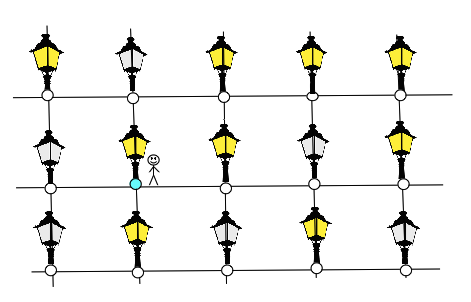}
\caption{Vertices of the lamplighter group $\mathbb{Z}^2 \wr \mathbb{Z}/2\mathbb{Z}$ are pairs: a position of the walker in $\mathbb{Z}^2$ (here shown with a  cyan circle) and a  configuration  of lamps, with finite number of lamps litten. The edges of the graph correspond either to steps of the walker in $\mathbb{Z}^2$ or to switching the lamp in the position of the walker.}
\label{lamplighter3}
\end{figure}

It is well known that the word length in a lamplighter group is closely
related to the Travelling Salesman Problem (TSP) in the base group,
see e.g. \cite{parry}. Therefore it is straightforward that the TSP
arises in the study of the drift function of random walks. Basic estimates
for the TSP can be useful to understand the behavior of random walks
on lamplighter groups, in particular the asymptotics of the drift functions.
We also mention that questions in metric geometry, in particular around
distortions of embeddings of wreath products into Banach spaces, require
deeper understanding of TSP. In Naor and Peres \cite{NaorPeres},
the Jones traveling salesman theorem \cite{jones} is applied to construct
embeddings of lamplighter groups into Banach spaces.

Denote an element $g$ of the lamplighter group $G=\mathbb{Z}^{2}\wr(\mathbb{Z}/2\mathbb{Z})$
by $(x,f)$, where $x\in\mathbb{Z}^{2}$ and $f:\mathbb{Z}^{2}\to\mathbb{Z}/2\mathbb{Z}$
is a function of finite support which we refer to as the lamp configuration
of $g$. 
Consider the standard generating set $S=\left\{ s_{1}=(e_{1},{\bf 0}),s_{2}=(e_{2},{\bf 0}),\delta=(0,\delta_{0}^{1})\right\} $,
where $e_{1}=(1,0)$, $e_{2}=(0,1)$ form the standard basis of $\mathbb{Z}^{2}$,
${\bf 0}$ is the constant function on $\mathbb{Z}^{2}$ that takes
value $0$, and $\delta_{0}^{1}$ is $0$ everywhere in $\mathbb{Z}^{2}$
except that it takes value $1$ at $0\in\mathbb{Z}^{2}$. Given a
finite subset $D\subseteq\mathbb{Z}^{2},$ denote by ${\rm \ell}_{{\rm TS}}(D)$
the length of a shortest path in $\mathbb{Z}^{2}$ (with respect to
the standard generating set $\left\{ e_{1},e_{2}\right\} $) that
visits every point in $D$. One of the usual settings for the TSP
also specifies the beginning and ending points of the path. Here we
do not specify the beginning and ending points (in the situations
we consider the diameter of $D$ is much smaller than $\ell_{{\rm TS}}(D)$).
Then the word length of a group element $g=(x,f)$ satisfies 
\[
{\rm {\rm \ell}_{{\rm TS}}}\left({\rm supp}f\right)+|{\rm supp}f|\le l_{S}(g)\le{\rm {\rm \ell}_{{\rm TS}}}\left({\rm supp}f\right)+|{\rm supp}f|+|x|+2{\rm Diam}({\rm supp}f),
\]
where $\left|X\right|$ is the cardinality of a set $X$, and ${\rm supp}f=\{x\in\mathbb{Z}^{2}:f(x)\neq0\}$. 

Consider the standard switch-walk-switch (SWS) measure $\mu=\eta\ast\nu\ast\eta$
on $\mathbb{Z}^{2}\wr\left(\mathbb{Z}/2\mathbb{Z}\right)$, where
$\eta$ is uniform on $\left\{ id_{G},\delta\right\} $ and $\nu$
is uniform on $\left\{ s_{1}^{\pm1},s_{2}^{\pm1}\right\} $. We refer
to the random walk $\left(X_{n}\right)_{n=0}^{\infty}$ with step
distribution $\mu$ as the standard SWS random walk on $\mathbb{Z}^{2}\wr\left(\mathbb{Z}/2\mathbb{Z}\right)$.
Write $X_{n}=\left(\bar{X}_{n},\Phi_{n}\right)$. Denote by $\mathcal{R}_{n}$
the range of the projected random walk $\left(\bar{X}_{n}\right)_{n=0}^{\infty}$
on $\mathbb{Z}^{2}$ up to time $n$. It is easy to see that $l_{S}(X_{n})$
is equivalent to the size of the range $\mathcal{R}_{n}$, up to a
multiplicative constant. Our goal is to determine the asymptotics
of $l_{S}(X_{n})/(n/\log n)$ when $n\to\infty$. Given a parameter
$p\in(0,1)$, the set obtained by keeping each point in $\mathcal{R}_{n}$
independently with probability $p$ is called the diluted range with
parameter $p$. Then the distribution of ${\rm supp}\Phi_{n}$ of
the standard SWS random walk is the same as the diluted range $\mathcal{R}_{n}$
with parameter $1/2$. Therefore we are led to consider the ${\rm TSP}$
of the diluted range.

Let $F_{n}$ be the length of a shortest path that visits the sites of a diluted
square two-dimensional lattice of side length $n$ with parameter
$p$, then by subadditivity, $\mathbb{E}[F_{n}]/n^{2}$ converges
to a constant $\alpha_{p}$ when $n\to\infty$, see Lemma \ref{dilutedlemma1}.
Indeed, the stronger statement that $F_{n}/n^{2}$ converges to $\alpha_{p}$
almost surely is true: this observation probably goes back to Beardwood,
Halton and Hammersley \cite{BHH}. The TSP on diluted lattice was
considered by Chakrabarti \cite{chakrabarti} and Dhar et al \cite{dharetal}
on the triangular lattice and the square lattice, who estimated the
constant $\alpha_{p}$ in terms of the percolation parameter $p$.

We estimate the length of ${\rm TSP}$ of the diluted range, claiming that this length, normalised by the size of the range, converges almost surely to a positive constant, see Lemma \ref{dilutedlemma2}. To show this claim
we subdivide
$\mathbb{Z}^{2}$ into boxes of certain side length $C$ and inside
each each box use the uncrossing Lemma \ref{uncross}. Part (i) of
Lemma \ref{uncross} is similar the circle freeway lemma used by Lalley
in \cite{lalley} for the TSP of random points on $\mathbb{R}^{2}$
with self-similar distribution; and in part (ii) we estimate the number
of \textquotedbl uncrossings\textquotedbl . To control the error of
approximations, we rely on the F{\o}lner property of the range process
$\left(\mathcal{R}_{n}\right)_{n=1}^{\infty}$. Given a finite $V\subset\mathbb{Z}^{2}$,
denote by $\partial V$ the inner boundary of $V$, that is, the set
of points in $V$ which have at least one neighbor site outside $V$.
We say that a sequence of finite sets $\left(V_{n}\right)_{n=1}^{\infty}$
forms a a F{\o}lner sequence if $\left|\partial V_{n}\right|/\left|V_{n}\right|\to0$
as $n\to\infty$. Consider the inner boundary $\partial\mathcal{R}_{n}$
of the random walk range $\mathcal{R}_{n}$. In Okada \cite{okada}
it is shown that there exists a constant $c\in\left[\pi^{2}/2,2\pi^{2}\right]$
such that $\lim_{n\to\infty}\frac{\mathbb{E}\left[\left|\partial\mathcal{R}_{n}\right|\right]}{n/\log^{2}n}=c$.
The property that almost surely $\left(\mathcal{R}_{n}\right)_{n=1}^{\infty}$
forms a F{\o}lner sequence is proved for a finite range symmetric
random walk on $\mathbb{Z}^{2}$ in Deligiannidis and Kosloff \cite{DK};
and extended to symmetric random walks on $\mathbb{Z}^{2}$ with finite
second moment in Deligiannidis, Gou{\"{e}}zel and Kosloff \cite{DGK}.
More precisely, by \cite[Theorem 12]{DGK}, for a centered, non-degenerate
random walk of finite second moment on $\mathbb{Z}^{2}$, there exists
a constant $c\in(0,\infty)$ such that almost surely 
\[
\lim_{n\to\infty}\frac{\left|\partial\mathcal{R}_{n}\right|}{n/\log^{2}n}=c.
\]

The proof for general random walks as in the statement of Theorem
\ref{main}, not necessarily SWS random walks, follows a similar outline,
and we explain below additional arguments we have to use in this general
case. We need in particular to have an analog of a \textquotedbl simple\textquotedbl{}
lemma for values of configurations in the wreath product, rather than
for diluted squares (see Lemma \ref{dilutedS1}); and a more general
form of \textquotedbl the  Uncrossing Lemma\textquotedbl{} in terms of wreath
products (see Lemma \ref{uncross2}). We also need to control the
closeness of the values of the configuration to the i.i.d. on the
range, see Lemma \ref{closeuniform}, in this step we use a result
of Flatto \cite{flatto}. The proof of the main result  of that paper is written for
simple random walk, but the proof works for centered random walks
with $(2+\epsilon)$-moment, see \cite[Remark after Theorem 3.1]{flatto}.
This is the only place we need the finite $(2+\epsilon)$-moment assumption
rather than finite second moment. It is natural to ask whether the
statement of \cite{flatto} is true under finite second moment assumption,
but to our knowledge it is not known. Finally, we mention one more
extra ingredient what we need for the proof of Theorem \ref{main}
in the case when the measure is infinitely supported. We need to control
large jumps between visited squares of subdivision, see the proof
of upper bound for $\left|X_{n}\right|$ at the end of Section \ref{sec:The-general-case}
(here the assumption of finite second moment is sufficient).

\section{Proof of theorem \ref{main} in the case of standard SWS random walk
and standard generating set\label{sec:standard}}

In this section we consider the case where $S$ is the standard generating
set of $G=\mathbb{Z}^{2}\wr\left(\mathbb{Z}/2\mathbb{Z}\right)$,
$S=\left\{ s_{1}=(e_{1},{\bf 0}),s_{2}=(e_{2},{\bf 0}),\gamma=(0,\delta_{0}^{1})\right\} $;
as we have mentioned in the introduction, $e_{1},e_{2}$ are standard
generators of $\mathbb{Z}^{2}$ and $\delta_{0}^{1}:\mathbb{Z}^{2}\to\mathbb{Z}/2\mathbb{Z}$
is the delta function at $(0,0)$, defined by $\delta_{0}^{1}(0,0)=1$
and $\delta_{0}^{1}(x)=1$ if $x\ne(0,0)$. We consider random walk
step distribution $\mu$ is the switch-walk-switch measure $\eta\ast\nu\ast\eta$,
where $\eta$ is uniform on $\left\{ id_{G},\delta\right\} $ and
$\nu$ is uniform on $\left\{ (\pm e_{1},{\bf 0}),(\pm e_{2},{\bf 0})\right\} $.
The projection of $\mu$ to $\mathbb{Z}^{2}$ is $\nu$, the standard
simple random walk on $\mathbb{Z}^{2}$.

\subsection{First observations about range of simple random walks on $\mathbb{Z}^{2}$}

Given a finite set $V\subseteq\mathbb{Z}^{2}$, recall that we denote
by $l_{{\rm TS}}(V)$ the length of a shortest path in the Cayley
graph of $\mathbb{Z}^{2}$ with respect to generators $e_{1}$, $e_{2}$
that visits all vertices in $V$. Note the following property of connected
sets in $\mathbb{Z}^{2}$.

\begin{lemma}\label{TSP1}

Let $V$ be a connected subset of $\mathbb{Z}^{2},$then 
\[
{\rm \ell_{TS}}(V)\le|V|\left(1+8\left(\frac{|\partial V|}{|V|}\right)^{1/3}\right).
\]

\end{lemma}

\begin{proof}

For any positive integer $C$, subdivide $\mathbb{Z}^{2}$ into squares
of size $C\times C$. Consider only those squares that intersect with
$V$ and denote by $T$ the union of the perimeters of these squares.
Denote by $V'$ the subset of $V$ which consists of points whose
squares are not completely contained in $V$. Note that since by assumption
$V$ is a connected subset of $\mathbb{Z}^{2}$, we have that $T$
is connected as well. We have that 
\begin{align*}
\left|T\right| & \le4C\left(\frac{\left|V\right|}{C^{2}}+\left|\partial V\right|\right),\\
\left|V'\right| & \le C^{2}\left|\partial V\right|.
\end{align*}
Let $\gamma_{{\rm TS}}(T)$ be a shortest path that goes through all
edges in $T$. One can then visit all points in $V$ by going along
this path through $T$ in the following way. \textcolor{black}{We
follow the path $\gamma_{{\rm TS}}(T)$ that visits points of $T$.}
For each $C\times C$ square $S$, look at the first visit of its
boundary by $\gamma_{{\rm TS}}(T)$. We stop at this point, insert
a path that visits all points of $V\cap S$, return along $T$ to
the same point of the boundary of $S$ where the path inside $S$
is inserted, and then continue along the path $\gamma_{{\rm TS}}(T)$.
Such  path provides an upper bound for ${\rm \ell_{TS}}(V$ 
\begin{align*}
{\rm \ell_{TS}}(V) & \le\left|\gamma_{{\rm TS}}(T)\right|+\left(4C+C^{2}\right)\frac{\left|V\right|}{C^{2}}+2|V'|\\
 & \le|V|+8\frac{\left|V\right|}{C}+\left(4C+2C^{2}\right)|\partial V|.
\end{align*}
Choosing that $C>0$ such that $\left(C+1\right)^{3}=4|V|/\left|\partial V\right|$,
we obtain the statement of lemma. 
\end{proof}

For the $\nu$-random walk on $\mathbb{Z}^{2}$, by \cite[Theorem 4.1]{DK},
the sequence $\left(\mathcal{R}_{n}\right)_{n=1}^{\infty}$ is almost
surely a F{\o}lner sequence. Thus we have the following:

\begin{corollary}

Consider a standard simple random walk on $\mathbb{Z}^{2}$. Let $\mathcal{R}_{n}$
be range of the random walk up to time $n$, then 
\[
\lim_{n\to\infty}\frac{\ell_{{\rm TS}}\left(\mathcal{R}_{n}\right)}{\left|\mathcal{R}_{n}\right|}=1\mbox{ a.s.}
\]

\end{corollary}

\begin{proof}

By \cite[Theorem 4.1]{DK}, $\left|\partial R_{n}\right|/\left|R_{n}\right|\to0$
a.s. when $n\to\infty$, thus for $\mathbb{P}$-a.e. $\omega$, there
exists a number $N(\omega)=N_{\epsilon}(\omega)$ such that for any
$n\ge N$, we have $\left|\partial R_{n}\right|/|R_{n}|<\epsilon$.
It follows that from Lemma \ref{TSP1}, for $n>N(\omega)$, $R_{n}^{{\rm TSP}}\le(1+8\epsilon^{1/3})|R_{n}|$.
In the lower bound direction, $R_{n}^{{\rm TSP}}\ge|R_{n}|$, the
statement follows.

\end{proof}

\subsection{An auxiliary fact about TSP on the diluted lattices}

Consider a square of size $n\times n$ in $\mathbb{Z}^{2}$ and denote
by $D_{n}$ the diluted lattice (site percolation) where each vertex
of the lattice is present independently with probability $p$. By
almost sub-additivity we have the following.

\begin{lemma}\label{dilutedlemma1}

There exists a constant $\alpha_{p}>0$ such that $\mathbb{E}\left[\ell_{{\rm TS}}\left(D_{n}\right)\right]/n^{2}\to\alpha_{p}$
when $n\to\infty$.

\end{lemma}

\begin{proof}

We first show the convergence along the subsequence $\left(2^{m}\right)_{m=1}^{\infty}$.
Consider a square of size $2^{m+1}\times2^{m+1}$ and subdivide it into
$4$ squares  $C_1$, $C_2$, $C_3$, $C_4$.. In each square $C_{i}$, $1 \le i \le 4$, 
 take a traveling salesman path
$P_{i}$ visiting the points in the diluted lattice in $C_{i}$. We denote by 
$A_{i}$ and $B_{i}$ the starting and ending point of $P_{i}$.
Then one can find a path visiting all points in the diluted lattice
of the larger square by following the paths $P_{1},P_{2},P_{3}$ and
$P_{4}$, adding in  shortest paths connecting $B_{1}$ to $A_{2}$,
$B_{2}$ to $A_{3}$ and $B_{3}$ to $A_{4}$. Therefore 
\[
\mathbb{E}\left[\ell_{{\rm TS}}\left(D_{2^{m+1}}\right)\right]\le4\mathbb{E}\left[\ell_{{\rm TS}}\left(D_{2^{m}}\right)\right]+9\cdot2^{m}.
\]
Write $b_{m}=\mathbb{E}\left[\ell_{{\rm TS}}\left(D_{2^{m}}\right)\right]/2^{2m}$,
then we have that $b_{m+1}\le b_{m}+9\cdot2^{-m-2}$.
% In particular, 
It follows that 
the sequence $b_{m}$ 
%is almost monotone and therefore
converges to
a limit constant $\alpha_{p}$. Since $\mathbb{E}\left[\left|D_{2^{m}}\right|\right]\ge p2^{2m}$,
we have that $\alpha_{p}\ge p>0$.

For general $n$, take $2^{j}$ such that $2^{j}\ll n$. Consider
the subsquare of side length $2^{j}\left\lfloor n/2^{j}\right\rfloor $
inside the square of side length $n$ with the same lower left corner.
Then we have 
\begin{equation}
\mathbb{E}\left[\ell_{{\rm TS}}\left(D_{n}\right)\right]\le\mathbb{E}\left[\ell_{{\rm TS}}\left(D_{2^{j}\left\lfloor n/2^{j}\right\rfloor }\right)\right]+4\cdot2^{j}n.\label{eq:dyadic1}
\end{equation}
To get a lower bound, for $2^{k}\gg n$, divide its sub-square of
side length $n\left\lfloor 2^{k}/n\right\rfloor $ into squares of
side length $n$ and connecting the TSP inside each small square,
we have 
\begin{equation}
\mathbb{E}\left[\ell_{{\rm TS}}\left(D_{2^{k}}\right)\right]\le\left(\frac{2^{k}}{n}\right)^{2}\mathbb{E}\left[\ell_{{\rm TS}}\left(D_{n}\right)\right]+C2^{2k}/n.\label{eq:dyadic2}
\end{equation}
It follows that $\mathbb{E}\left[\ell_{{\rm TS}}\left(D_{n}\right)\right]/n^{2}$
converges to $\alpha_{p}$ as $n\to\infty$.

\end{proof}

\begin{remark}

It is not hard to see that $p<\alpha_{p}<1$. Indeed,  observe that a positive frequency
of absent squares of size $C\times C$, for a large constant $C$,
implies that $\alpha_{p}<1$. On the other hand, positive frequencies
of hanging edges show that $\alpha_{p}>p$. The Bartholdi-Platzman
bound from \cite{BP82} (which is recalled in Lemma \ref{multiscale})
implies that $\lim_{p\to0+}\alpha_{p}=0$.

\end{remark}

\subsection{An uncrossing lemma\label{subsec:uncrossing1}}

The key ingredient in our argument is the following uncrossing lemma.
By a path in $\mathbb{Z}^{2}$ we mean a directed path in the standard
Cayley graph of $\left(\mathbb{Z}^{2},(e_{1},e_{2})\right)$. Given
a set $D$ in $\mathbb{Z}^{2}$ and two paths $P_{1}$, $P_{2}$ where
each $P_{i}$ starts and terminates at the boundary $\partial D$,
write $A_{i}$ for the starting point of $P_{i}$ and $B_{i}$ for
the end point of $P_{i}$. We say $P_{1}$ and $P_{2}$ have an \emph{essential
crossing} if $A_{2}$ and $B_{2}$ are contained in the two different
open arcs of $\partial D$ between $A_{1}$ and $B_{1}.$ 
%See the picture on the left of Figure \ref{uncrossing0}. ???????

Given a finite collection $\mathcal{T}$ of paths starting and ending
at the boundary $\partial D$, we perform the following procedure.
If there are two paths $P_{1}$ and $P_{2}$ with a common starting/ending
point such that $\left\{ A_{1},B_{1}\right\} \cap\left\{ A_{2}, B_{2}\right\} \neq\emptyset$,
then reversing the direction of one of them if necessary, we replace
them by the path which is the concatenation of $P_{1},P_{2}$. By
performing this repeatedly, we can replace $\mathcal{T}$ by $\mathcal{T}'$
in such a way  that the starting and ending points of the paths in the new collection
$\mathcal{T}'$ are pairwise disjoint. In $\mathcal{T}'$ each starting
or ending point appears exactly once, unless there is a path with
the same starting/ending point (that is, a loop). Note that by definition,
a loop does not have essential crossing with any other path.

Since we consider the standard generating set $\{e_{1},e_{2}\}$ of
$\mathbb{Z}^{2}$, if two paths $P_{1}$ and $P_{2}$ have essential
crossings then they must intersect. Let $O$ be an intersection point
of $P_{1}$ and $P_{2}$. We call the following operation \emph{uncrossing
of $P_{1}$ and $P_{2}$ around $O$}: replace the path $P_{1}$ by
$Q_{1}$, where $Q_{1}$ is the concatenation of the segment of $P_{1}$
from $A_{1}$ to $O$ and the segment of $P_{2}$ from $O$ to $B_{2}$;
and similarly replace $P_{2}$ by $Q_{2}$ which concatenates the
segment of $P_{2}$ from $A_{2}$ to $O$ and the segment of $P_{1}$
from $O$ to $B_{1}$. It is clear that the union of images of $Q_{1}$
and $Q_{2}$ is the same as  the  union of $P_{1}$ and $P_{2}$, and,  moreover,
$Q_{1}$ and $Q_{2}$ do not have an essential crossing. See Figure
\ref{uncrossing0}. Note that $Q_{1}$ and $Q_{2}$ might have intersection
points other than $O$.

\begin{figure}
\includegraphics{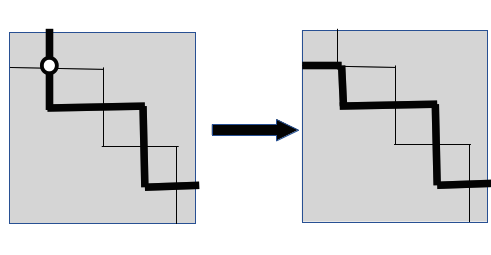} \caption{Uncrossing two paths}
\label{uncrossing0} 
\end{figure}
Apply this uncrossing procedure repeatedly, we have the following
lemma. The first part is essentially the same as {\it the circle freeway
lemma} in Lalley \cite{lalley}.

\begin{lemma}[Uncrossing lemma]\label{uncross}

Consider a collection $\mathcal{T}$ of paths in a finite domain $D$
where each path starts and ends on the boundary of $D$. 
\begin{description}
\item [{(i)}] There is another collection $\mathcal{S}$ of paths, which
has the same union of their images as the initial one, and no two
paths in the new collection has an essential crossing. 
\item [{(ii)}] One can obtain such a collection $\mathcal{S}$ from the
original collection $\mathcal{T}$ by first replacing $\mathcal{T}$
by $\mathcal{T}'$ such that the paths in $\mathcal{T}'$ have pairwise
disjoint starting and ending points, and then performing at most $m$
uncrossings, where $m$ is the number of paths in $\mathcal{T}'$. 
\end{description}
\end{lemma}

\begin{proof}

Given a finite collection of paths $\mathcal{T}$ in $D$, we first
replace $\mathcal{T}$ by $\mathcal{T}'$ via joining paths (reversing
direction of some of them if necessary) as explained before.

We prove the statement by induction on the number of paths of in the
collection $\mathcal{T}'$. Given a path $P$ with starting point
$A$ and ending point $B$, denote by $C_{1},C_{2}$ the two arcs
on $\partial D$ between $A$ and $B$. We say that a path $P$
in the collection $\mathcal{T}'$ is good if there exists  $i$, $i=1$ or $i=2$, 
such that all other paths in $\mathcal{T}'$ have starting and ending
points in $C_{i}$. In particular it implies that $P$ has no essential
crossings with the other paths. If there is a good path $P$ in $\mathcal{T}',$
then we can remove it and apply the induction hypothesis to the collection
$\mathcal{T}'\setminus\{P\}$ consisting of the rest of paths. Note
that performing uncrossing for the rest of the paths will not introduce
any essential crossing with $P$. If there is no good path, pick any
path $P$ and on the arc $C_{1}$ take the point $A'$ which is the
closest to $A$ which is starting or ending point some other path
$Q$ in the collection. Reverse the direction of $Q$ if necessary,
we uncross $P$ and $Q$ such that one of the new path starts at $A$
and end at $A'$. Now this new path is good, we remove it and apply
the induction hypothesis.

\end{proof}

After performing uncrossings to the original collection $\mathcal{T}$,
we can join the paths in the new collection $\mathcal{S}$.

\begin{lemma}\label{joining1}

Suppose $\mathcal{S}$ is a collection of paths in the domain $D$
such that each path starts and ends on $\partial D$ and there is
no essential crossings between paths of $\mathcal{S}$. Then there
is a path of total length at most $1.5\left|\partial D\right|+\sum_{P\in\mathcal{S}}|P|$
whose image is obtained from the union of paths in $\mathcal{S}$ by adding
connecting segments along the perimeter of $D$.

\end{lemma}

\begin{figure}
\includegraphics{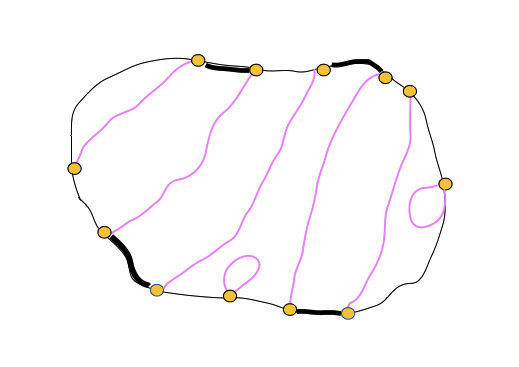} \caption{First we ignore possible loops, pass through all other paths in the
collection, adding intervals on the perimeter of the domain (shown
in black). The total sum of these black intervals can be chosen to
be at most one half of the perimeter.}
\label{uncrossed1} 
\end{figure}
\begin{proof}

First we pass through all paths that are not loops (see Figure \ref{uncrossed1}).
The length of this path is at most the sum of the lengths of non-loop
paths plus $|\partial D|/2$. Then we go through all the loops, passing  through
any point of the perimeter at most once. It is clear that  the total sum of the
segments of the perimeter we  passed  through is at most the length of the perimeter.
Therefore the total length of our path is at most sum of the lengths
of initial paths plus $1.5|\partial D|$.

\end{proof}

The multiplicative constant $1.5$ in front of $\left|\partial D\right|$
is not important for our applications, but it is optimal in the setting
of Lemma \ref{joining1}.

\subsection{Standard random walk on $\mathbb{Z}^{2}\wr(\mathbb{Z}/2\mathbb{Z})$}

Denote by $\mathcal{DR}_{n}$ the diluted range $\mathcal{R}_{n}$
with parameter $1/2$. Recall that $\ell_{{\rm TS}}(\mathcal{DR}_{n})$
denotes the shortest length of a path visiting all points in the diluted
range $\mathcal{DR}_{n}$. By using the Uncrossing Lemma \ref{uncross}, we show
that $\ell_{{\rm TS}}(\mathcal{DR}_{n})/|\mathcal{R}_{n}|$ converges
to a positive constant almost surely. The steps in the proof are illustrated
in Figures \ref{TravelingSP}, \ref{uncrossed2} and \ref{join}.

\begin{figure}
\includegraphics[scale=0.95]{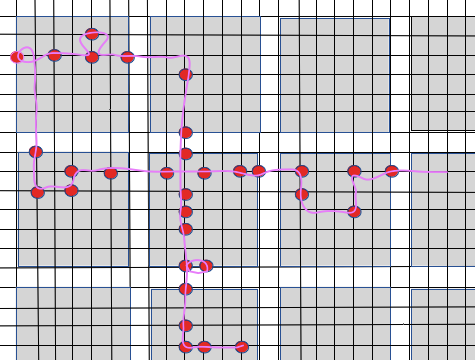} \caption{We subdive a path which visits the red nodes into intervals that stay
inside each grey square of side length $C$.}
\label{TravelingSP} 
\end{figure}
\begin{lemma}\label{dilutedlemma2}

The TSP length of the diluted range satisfies 
\[
\frac{\ell_{{\rm TS}}(\mathcal{DR}_{n})}{|\mathcal{R}_{n}|}\to\alpha_{1/2}\mbox{ a.s. when }n\to\infty,
\]
where the constant $\alpha_{1/2}$ is the same as in Lemma \ref{dilutedlemma1}
with $p=1/2$.

\end{lemma}

\begin{figure}
\includegraphics[scale=1.1]{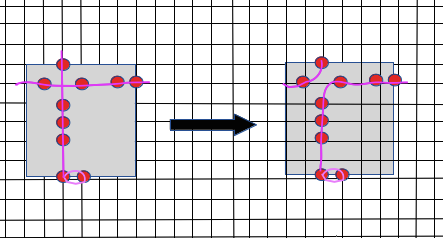} \caption{Given a collection of paths inside a grey square, we perform the uncrossing
procedure.}
\label{uncrossed2} 
\end{figure}
\begin{figure}
\includegraphics[scale=1.4]{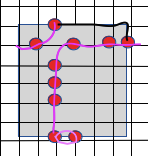} \caption{Given a collection of (non-oriented) paths (magenta coloured paths
on the picture) without essential crossings, we join them together
in one single path, adding new paths in the boundary of the square
(on this picture there is only one additional path, shown in black).
The absence of essential uncrossings guarantees that the total length
of new paths is at most the perimeter of the square (that is, at most
4C).}
\label{join} 
\end{figure}
\begin{proof}

Take a large integer $C$ and divide $\mathbb{Z}^{2}$ into boxes
of side length $C$. Consider those boxes with non-empty intersection
with the range $\mathcal{R}_{n}$.

We first show an upper bound on $\ell_{{\rm TS}}(\mathcal{DR}_{n})$
by describing a path that visits all points in $\mathcal{DR}_{n}$.
Take a path that consists of the union $T_{n}$ of the perimeters
of the boxes with non-empty intersection with $\mathcal{R}_{n}$.
We say a box $\alpha$ is fully visited if $\alpha\subseteq\mathcal{R}_{n}$.
Suppose $\alpha$ is a box not full, we take a path that starts at
a point of $\alpha\cap\partial\mathcal{R}_{n}$, visits all points
of $\mathcal{R}_{n}\cap\alpha$ and returns to its staring point. The
length of such path is at most $2|\mathcal{R}_{n}\cap\alpha|$.
If a box $\alpha$ is full, then inside $\alpha$ we are in the situation
of TSP of a diluted lattice. Therefore, to visit all points in $\mathcal{D}\mathcal{R}_{n}$,
we can go along the perimeter $T_{n}$, and in each box $\alpha$
we take a shortest path visiting all points inside $\mathcal{DR}_{n}|_{\alpha}$
then return to $T_{n}$. The length of such a path is bounded by 
\begin{equation}
\ell_{{\rm TS}}(\mathcal{DR}_{n})\le\sum_{\alpha\mbox{ full}}\ell_{\mathrm{TS}}\left(\mathcal{DR}_{n}|_{\alpha}\right)+\sum_{\alpha\ \mbox{not full}}2|\mathcal{R}_{n}\cap\alpha|+2|T_{n}|.\label{eq:upper1}
\end{equation}
If a box is not full, then it must contain some point in $\partial\mathcal{R}_{n}$.
Therefore the number of boxes which are  not full is bounded by $\left|\partial\mathcal{R}_{n}\right|$
and 
\[
\sum_{\alpha\ \mbox{not full}}|\mathcal{R}_{n}\cap\alpha|\le C^{2}|\partial\mathcal{R}_{n}|.
\]
Similarly the length of the perimeter union $T_{n}$ is bounded by
\[
|T_{n}|\le4|\partial\mathcal{R}_{n}|+4\left|\mathcal{R}_{n}\right|/C.
\]
Plugging these bounds into (\ref{eq:upper1}), we obtain
\begin{equation}
{\rm \ell_{{\rm TS}}}\left(\mathcal{DR}_{n}\right)\le\sum_{\alpha\mbox{ full}}{\rm \ell_{{\rm TS}}}\left(\mathcal{DR}_{n}|_{\alpha}\right)+(2C^{2}+4)|\partial\mathcal{R}_{n}|+4|\mathcal{R}_{n}|/C.\label{eq:upper2}
\end{equation}
The main contribution to the length of the path comes from  the TSP paths
in full boxes. Denote by $D_{C}$ the diluted lattice obtained from
a $C\times C$ square where each vertex of the square is present independently
with probability $1/2$. By the law of large numbers, 
\begin{equation}
\frac{\sum_{\alpha\mbox{ full}}{\rm \ell_{{\rm TS}}}\left(\mathcal{DR}_{n}|_{\alpha}\right)}{C^{2}\cdot\#\mbox{ of full boxes in }\mathcal{R}_{n}}\to\frac{\mathbb{E}\left[\ell_{{\rm TS}}\left(D_{C}\right)\right]}{C^{2}}\mbox{ a.s. when }n\to\infty.\label{eq:limitfull}
\end{equation}

Now we turn to the lower bound. Take a global TSP path $P$ in the standard generating set 
$\left\{ \pm e_{1},\pm e_{2}\right\} $ of $\mathbb{Z}^2$
which visits all the points
in the diluted range $\mathcal{D}\mathcal{R}_{n}$. Consider a box
$\alpha$ that is fully visited and restrict the path $P$ to $\alpha$.
We obtain a collection of connected paths where each path starts and
ends on the boundary of $\alpha$. We now apply the Uncrossing Lemma
\ref{uncross} to this collection of paths and connect the uncrossed
paths along the perimeter of $\alpha$, using Lemma \ref{joining1}.
It follows that there is a connected path of length at most $\left|P|_{\alpha}\right|+6C$
which visits all points in $\mathcal{DR}_{n}|_{\alpha}$. Therefore,
we have that for each fully visited box $\alpha$, 
\[
\left|P|_{\alpha}\right|\ge{\rm \ell_{{\rm TS}}}\left(\left|P|_{\alpha}\right|\right)-6C.
\]
Summing up over all fully visited box, we have 
\begin{equation}
{\rm \ell_{TS}}\left(\mathcal{DR}_{n}\right)\ge\sum_{\alpha\mbox{ full}}\left({\rm \ell_{TS}}\left(\mathcal{DR}_{n}|_{\alpha}\right)-6C\right).\label{eq:lower1}
\end{equation}

Finally we combine the upper and lower bounds. Given any $\epsilon>0$,
by Lemma \ref{dilutedlemma1}, we can take a constant $C>100/\epsilon$
sufficiently large such that 
\[
\left|\frac{\mathbb{E}\left[\ell_{\mathrm{TS}}\left(D_{C}\right)\right]}{C^{2}}-\alpha_{1/2}\right|<\epsilon/100.
\]
Then by combining (\ref{eq:upper2}), (\ref{eq:limitfull}) and (\ref{eq:lower1})
and applying the result that  $\mathcal{R}_{n}$ forms
a F{\o}lner sequence almost surely, we conclude that almost surely, 
\[
\limsup_{n\to\infty}\left|\frac{{\rm \ell_{TS}}(\mathcal{DR}_{n})}{|\mathcal{R}_{n}|}-\alpha_{1/2}\right|<\epsilon.
\]
Since $\epsilon>0$ is arbitrary, we have proved the statement.

\end{proof}

\begin{proof}[Proof of Theorem \ref{main} for standard SWS measure
and standard generating set]

Recall that $(X_{n})_{n=0}^{\infty}$ denotes a $\mu$-random walk
on $G$ and we write $X_{n}=\left(\bar{X}_{n},\Phi_{n}\right)$. The
distribution of support of $\Phi_{n}$ is the same as the diluted
range $\mathcal{DR}_{n}$ with parameter $p=1/2$. For the group element
$X_{n}$, a path in $S\cup S^{-1}$ that connects $id_{G}$ to $X_{n}$
satisfies that its projection to the base $\mathbb{Z}^{2}$ visits
all the points in ${\rm supp}\Phi_{n}$ and in the end arrives at
$\bar{X}_{n}$; and along the way switch on the lamps in ${\rm supp}\Phi_{n}$.
Therefore, we have that 
\[
{\rm \ell_{TS}}({\rm supp}\Phi_{n})+|{\rm supp}\Phi_{n}|\le\ell_{S}(X_{n})\le{\rm \ell_{TS}}({\rm supp}\Phi_{n})+|{\rm supp}\Phi_{n}|+3\max_{1\le t\le n}|\bar{X}_{t}|_{\bar{S}}.
\]
Recall that ${\rm supp}\Phi_{n}$ can be identified with $\mathcal{DR}_{n}$.
Since $\max_{1\le t\le n}|\bar{X}_{t}|_{\bar{S}}/(n/\log n)\to0$
almost surely when $n\to\infty$, by Lemma \ref{dilutedlemma2} and
the Dvoretzky-Erd{\"{o}}s law of large numbers for the range (\ref{eq:erdosdovret}),
we have 
\begin{align*}
\lim_{n\to\infty}\frac{1}{n/\log n}\ell_{S}(X_{n}) & =\lim_{n\to\infty}\frac{1}{n/\log n}{\rm \ell_{{\rm TS}}}({\rm supp}\Phi_{n})+\lim_{n\to\infty}\frac{1}{n/\log n}|{\rm supp}\Phi_{n}|.\\
 & =\left(\alpha_{1/2}+\frac{1}{2}\right)\pi.
\end{align*}

\end{proof}

\section{Uncrossing lemma for wreath product $\mathbb{Z}^{2}\wr(\mathbb{Z}/2\mathbb{Z})$
and general generating set}

In this section we consider an arbitrary finite generating set $S$
of $G=\mathbb{Z}^{2}\wr(\mathbb{Z}/2\mathbb{Z})$. By an $S$-path
we mean a path in the right Cayley graph of $\left(G,S\right)$, denote
by $(x_{0},x_{1},\ldots,x_{k}),$where each $x_{i}\in G$ and $x_{i-1}^{-1}x_{i}\in S$.
Recall that an element in the wreath product is denoted by $(x,f)$,
where $x\in\mathbb{Z}^{2}$ and $f:\mathbb{Z}^{2}\to\mathbb{Z}/2\mathbb{Z}$
is a function of finite support. Since the lamp group is $\mathbb{Z}/2\mathbb{Z}$,
the configuration $f$ can be identified with the support of $f$,
that is, the finite set where the value of $f$ is $1$. Given a finite
set $U$ in $\mathbb{Z}^{2}$, consider an $S$-path $(x_{0},x_{1},\ldots,x_{k})$
in $\mathbb{Z}^{2}\wr(\mathbb{Z}/2\mathbb{Z})$ such that $x_{0}$
is of the form $\left(z,{\bf 0}\right)$ and the support of the lamp
configuration of $x_{k}$ is exactly $U$. Denote by ${\rm \ell}_{{\rm TS}}^{(G,S)}(U)$
the length of the shortest $S$-path with this property.

Analogous to Lemma \ref{dilutedlemma1}, we have the following convergence
of expectation of length of traveling salesman $S$-paths after scaling
for the diluted lattice in a square of side length $n$. Consider
a square of size $n\times n$ in $\mathbb{Z}^{2}$ with the lower
left corner at $(0,0)$. Take the diluted lattice $D_{n}$ with parameter
$p$ inside the square where each vertex of the lattice is present
independently with probability $p$.

\begin{lemma}\label{dilutedS1}

There exists a constant $\alpha_{p,S}>0$ such that $\mathbb{E}\left[{\rm \ell}_{{\rm TS}}^{(G,S)}(D_{n})\right]/n^{2}\to\alpha_{p,S}$
when $n\to\infty$.

\end{lemma}

\begin{proof}

Write $F_{n}=\mathbb{E}\left[{\rm \ell}_{{\rm TS}}^{(G,S)}(D_{n})\right]$.
For a square of size $2^{n+1}\times2^{n+1}$, subdivide it into $4$
squares. In each sub-square $C_{i}$, take a shortest $S$-path $P_{i}$
where the end point $B_{i}$ of $P_{i}$ is such that $B_{i}=(b_{i},f_{i}),$
${\rm supp}f_{i}$ is exactly the diluted lattice in $C_{i}$. Write
$\left(a_{i},{\bf 0}\right)$ for the starting point of $P_{i}$.
Then one can find an $S$-path visiting all points in the diluted
lattice of the original square by following the paths $P_{1},P_{2},P_{3}$
and $P_{4}$, joining their starting/ending points by the shortest
$S$-paths connecting $(b_{1},{\bf 0})$ to $(a_{2},{\bf 0})$, $(b_{2},{\bf 0})$
to $\left(a_{3},{\bf 0}\right)$ and $(b_{3},{\bf 0})$ to $(a_{4},{\bf 0})$.
Therefore 
\[
F_{2^{n+1}}\le4F_{2^{n}}+C2^{n},
\]
where the constant $C$ only depends on $S$. Write $b_{n}=F_{2^{n+1}}/2^{2n}$,
then $b_{n+1}\le b_{n}+C2^{-n-2}$. In particular, the sequence $b_{n}$
is almost monotone and therefore converges to a limit $\alpha_{p,\bar{S}}$.
Since $F_{2^{n}}\ge\mathbb{E}\left[\left|D_{n}\right|\right]=p2^{2n}$,
we have that $\alpha_{p}\ge p>0$. To extend the convergence to general
$n$, we argue in the same way to establish bounds of the form (\ref{eq:dyadic1})
and (\ref{eq:dyadic2}) as in Lemma \ref{dilutedlemma1}.

\end{proof}

Given an $S$-path $P=(x_{0},x_{1},\ldots,x_{\ell})$, denote by $\tau_{P}$
the difference between lamp configurations of $x_{\ell}$ and $x_{0}$,
that is, $\tau_{P}:=f_{x_{\ell}}-f_{x_{0}}$. Write $|P|=\ell$ for
the length of the path $P$. Given a collection $\mathcal{A}$ of
$S$-paths, we refer to the sum of configurations $\sum_{P\in\mathcal{A}}\tau_{p}$ as
the lamp configuration produced by $\mathcal{A}$.

\begin{remark}\label{reverse}

The fact that the lamp group is $\mathbb{Z}/2\mathbb{Z}$ allows us
to reverse the direction of paths. If we reverse the direction of
$P$ to the path $\overleftarrow{P}=(x_{\ell},x_{\ell-1},\ldots,x_{1},x_{0})$,
then the path $\overleftarrow{P}$ makes the same changes in the lamp
configuration as $P$, while in the projection to $\mathbb{Z}^{2}$
the starting and ending points are swapped. This is only valid when
the lamp group is $\mathbb{Z}/2\mathbb{Z}$.

\end{remark}

Next we explain how to perform uncrossings to $S$-paths. Recall that
in Section \ref{sec:standard} we have defined essential crossings
for standard paths in $\mathbb{Z}^{2}$. The notion can be extended
to $S$-paths. Two $S$-paths are said to have \emph{essential crossings}
if their projections to $\mathbb{Z}^{2}$ have essential crossings.
We say two $\bar{S}$-paths $P_{1}$ and $P_{2}$ have an essential
crossing in $D$, where $P_{i}$ starts at $A_{i}$ and ends at $B_{i}$,
if $A_{1},A_{2},B_{1},B_{2}\in\partial D$, and $A_{2}$ and $B_{2}$
are contained in the two different open arcs of $\partial D$ between
$A_{1}$ and $B_{1}.$ Here $\partial D$ denotes the boundary of
$D$ with respect to the standard generating set $\{e_{1},e_{2}\}$.

The following lemma is a generalization of the uncrossing Lemma \ref{uncross}
to $S$-paths. Consider a configuration $f$ with ${\rm supp}f=U$.
Recall that for the standard generating set considered in Section
\ref{sec:standard}, the length of the shortest path that turns on
the prescribed configuration $f$, whose starting and ending point
in $\mathbb{Z}^{2}$ are $z_{1}$ and $z_{2}$ respectively, is exactly the
length of a traveling salesman path which starts at $z_{1}$, visits
every point in ${\rm supp}f$, and ends at $z_{2}$, plus $\left|{\rm supp}f\right|$.
And for the standard SWS random walk considered in Section \ref{sec:standard},
the configuration at time $n$ is the diluted range $\mathcal{DR}_{n}$
with parameter $1/2$. Now in the general case (as in the setting
of Lemma \ref{uncross2} below) we need to control the shortest path
to produce the configuration, where possible moves are associated
to the generating set $S$ (that is, using $S$-paths). For example,
these moves can be of the following form: first jump on $\mathbb{Z}^{2}$
at distance $2$ on the right, turn on the lamp at distance $5$ above
the marker on $\mathbb{Z}^{2},$ and then jump at distance $3$ down.
Any finite generating set $S$ can be interpreted as such rules, with
finitely many jumps and finitely many turnings on the lamps. Therefore
the following lemma is formulated in terms of wreath products (since
it uses $S$-paths) and gives the statement similar to Lemma \ref{uncross}
for an arbitrary generating set.

\begin{lemma}[Uncrossing lemma in terms of wreath product generators]\label{uncross2}

Let $D$ be an $C\times C$ square in $\mathbb{Z}^{2}$. Consider
a collection $\mathcal{A}$ of $S$-paths where the projection of
each path to $\mathbb{Z}^{2}$ starts and terminates on the boundary
of $D$. Then one can find another collection $\mathcal{B}$ of $S$-paths
such that there are no essential crossings between paths in $\mathcal{B}$,
each path in $\mathcal{B}$ starts and terminates at $\partial D$,
and $\mathcal{B}$ produces the same lamp configuration as $\mathcal{A}$.
Moreover the total length of $S$-paths in $\mathcal{B}$ is bounded
by 
\[
\sum_{Q\in\mathcal{B}}|Q|\le\sum_{P\in\mathcal{A}}|P|+c|\partial D|,
\]
where $c$ is a constant only depending on $S$.

\end{lemma}

\begin{proof}

Given an $S$-path $P=\left(x_{0},x_{1},\ldots,x_{k}\right)$, where
each increment $x_{i-1}^{-1}x_{i-1}\in S$, connect each pair $(x_{i-1},x_{i})$
by a shortest path in the standard generating set $\left\{ \delta,\left(e_{1}^{\pm1},{\bf 0}\right),\left(e_{2}^{\pm1},{\bf 0}\right)\right\} $.
Denote by $\hat{P}$ the new path in the standard generating set.
Given two $S$-paths $P_{1}$ and $P_{2}$ in a domain $D$, where
$P_{i}$ starts at $A_{i}$ and ends at $B_{i}$, we perform un uncrossing
as follows. First take the paths $\hat{P}_{1}$ and $\hat{P}_{2}$
in the standard generating set described above associated with $P_{1}$
and $P_{2}$. If $\hat{P}_{1}$ and $\hat{P}_{2}$ have essential
crossings, then we take an intersection point $O$ and perform the
uncrossing around $O$ as described in Subsection \ref{subsec:uncrossing1}.
This way we obtain two paths without essential crossing: $\hat{Q}_{1}$
which is the concatenation of the sub-path of $\hat{P}_{1}$ from
$A_{1}$ to $O$ and the sub-path of $\hat{P}_{2}$ from $O$ to $B_{2}$;
and $\hat{Q}_{2}$ which is the concatenation of the sub-path of $\hat{P}_{2}$
from $A_{2}$ to $O$ and the sub-path of $\hat{P}_{1}$ from $O$
to $B_{1}$. Now we convert the paths $\hat{Q}_{1},\hat{Q_{2}}$ in
the standard generating set back to $\bar{S}$-paths. If $O$ lies
on both $P_{1}$ and $P_{2}$, then let $Q_{1}$ be the concatenation
of the sub-path of $P_{1}$ from $A_{1}$ to $O$ and the sub-path
of $P_{2}$ from $O$ to $B_{2}$; and similarly for $Q_{2}.$ If
$O$ is not on $P_{1},$ then let $y_{1}$ be the  closest point
to $O$ which belongs to the intersection of $P_{1}$ and the sub-path
of $\hat{P}_{1}$ from $A_{1}$ to $O$. Similarly, let $z_{1}$ the
closest point to $O$ which belongs to the intersection of $P_{2}$
and the sub-path of $\hat{P}_{2}$ from $O$ to $B_{2}$. Connect
$y_{1}$ to $z_{1}$ by an $S$-path of shortest length. Then we define
$Q_{1}$ to be the concatenation of the $S$-path from $A_{1}$ to
$y_{1}$ along $P_{1}$, the $S$-path connecting $y_{1}$ to $z_{1}$
and the $S$-path from $z_{1}$ to $B_{2}$ along $P_{2}$. In the
same manner we obtain an $S$-path $Q_{2}$ from $A_{2}$ to $B_{1}$.

By the description above, it is clear that $\tau_{Q_{1}}+\tau_{Q_{2}}=\tau_{P_{1}}+\tau_{P_{2}}$
and moreover 
\[
|Q_{1}|+|Q_{2}|\le|P_{1}|+|P_{2}|+c,
\]
where the additive constant $c$ only depends on the generating set
$S$.

Consider a finite collection $\mathcal{A}$ of $S$-paths whose projection
to $\mathbb{Z}^{2}$ starts and ends at $\partial D$. If there are
two paths $P_{1}$ and $P_{2}$ such that $\left\{ A_{1},B_{1}\right\} \cap\left\{ A_{2}.B_{2}\right\} \neq\emptyset$,
then reversing the direction of one of them if necessary (this is
allowed because the lamp group is $\mathbb{Z}/2\mathbb{Z},$ see Remark
\ref{reverse}), we replace them by the path which is the concatenation
of $P_{1},P_{2}$. By performing this repeatedly, we can replace $\mathcal{A}$
by $\mathcal{A}'$ such that the starting/ending points of the paths
are pairwise disjoint. By the inductive argument in the proof of Lemma
\ref{uncross} (ii), the number of uncrossings we need to perform
is bounded by $\left|\partial D\right|$. Since each uncrossing costs
at most an additive constant $c$, this implies the statement of the
lemma.

\end{proof}

\section{The general case and proof of Theorem \ref{main}\label{sec:The-general-case}}

Throughout this section, let $\mu$ be a symmetric non-degenerate
probability measure on $G=\mathbb{Z}^{2}\wr(\mathbb{Z}/2\mathbb{Z})$
with finite second moment; and $S$ be a finite generating set of
$G$. By raising $\mu$ to a convolution power if necessary, we may
assume that $\mu(id)>0$ and $\mu(\delta)>0$, where $\delta=\left(\delta_{0}^{1},{\bf 0}\right)$.
Recall that $l_{S}$ denotes the word length with respect to the generating
set $S$.  Recall that one of the assumptions of Theorem \ref{main} is that
$\mu$ has $(2+\epsilon)$-moment for some $\epsilon>0$

Let $\pi$ be the projection $G\to\mathbb{Z}^{2}$.
Denote by $\bar{\mu}$ the projection of $\mu$ to $\mathbb{Z}^{2}$.
Recall that $\left(X_{n}\right)_{n=0}^{\infty}$ is a random walk
with step distribution $\mu$ on $G$ and we write $X_{n}=\left(\bar{X}_{n},\Phi_{n}\right)$,
where $\bar{X}_{n}=\pi\left(X_{n}\right)$ is the projection to $\mathbb{Z}^{2}$
and $\Phi_{n}$ is the lamp configuration of $X_{n}.$ Under our assumptions,
by \cite[Theorem 12]{DGK}, the range process $\left(R_{n}\right)_{n=0}^{\infty}$
of the $\bar{\mu}$-random walk on $\mathbb{Z}^{2}$ forms a sequence
of F{\o}lner  sets almost surely: more precisely, there exists a constant
$c>0$ such that 
\[
\lim_{n\to\infty}\frac{\left|\partial R_{n}\right|}{n/\log^{2}n}=c\mbox{ a.s.}
\]

We proceed by subdividing the $\mathbb{Z}^{2}$-lattice into boxes
of side length $c_{n}$, where $c_{n}$ is an integer depending on
$n$ (to be specified later). We say a $(c_{n}\times c_{n})$-box $\alpha$
is visited up to time instant $n$ if is $\alpha\cap R_{n}\neq\emptyset$.
Note that in a box $\alpha$ where every point has been visited by the
$\bar{\mu}$-random walk up to time $n$, the distribution of the
lamp configuration in $\alpha$ is in general not uniform on $\{0,1\}^{\alpha}$.
Instead we use that if the box is visited enough times, then the distribution of the lamp configuration 
is close to uniform and the TSP of the configuration can be controlled.
For this purpose we need the following two lemmas.

Write $a=\min\left\{ \mu(id),\mu(\delta)\right\} $ and $\mu=au+(1-a)\mu'$,
where $u$ is the uniform measure on $\{id,\delta\}$. Let $(Y_{n})_{n=1}^{\infty}$
be a sequence of i.i.d. Bernoulli random variables, with $\mathbb{P}(Y_{i}=1)=a$
and $\mathbb{P}\left(Y_{i}=0\right)=1-a$. Let $(Z_{n})_{n=1}^{\infty}$
be a sequence of a sequence of i.i.d. random variables with distribution
$\mu'$ and $\left(U_{n}\right)_{n=1}^{\infty}$ be a sequence of i.i.d.
random variables with distribution $u$. Then $ $ 
\[
\tilde{Z}_{n}=Z_{n}{\bf 1}_{\{Y_{n}=0\}}+U_{n}{\bf 1}_{\{Y_{n}=1\}}
\]
has distribution $\mu$. We think of the $\mu$-random walk increments
sampled in this way. Write $\tilde{W}_{n}=\tilde{Z}_{1}\ldots\tilde{Z}_{n}$.
We say  that a location $x\in\mathbb{Z}^{2}$ receives a good multiplication
up to time instant $n$ if there is a time $t\le n$ such that $\pi\left(\tilde{W}_{t-1}\right)=x$
and $Y_{t}=1$.

By deleting the steps with $Y_{n}=1$ in the sequence $\left(\tilde{Z}_{n}\right)$
we obtain a sequence of i.i.d. random variables with distribution
$\mu'$. Denote by $\left(W'_{t}\right)_{t=0}^{\infty}$ the random
walk on $G$ associated with this sequence of increments. We say a
box $\alpha$ is $q$-fully visited by the $\mu'$-walk by the time
$n$ if for every $x\in\alpha$, 
\[
\left|\left\{ t\le n:\pi\left(W_{t}'\right)=x\right\} \right|\ge q.
\]
Denote by $F_{n,q}^{\alpha}$ the event that the box $\alpha$ is
$q$-fully visited by the $\mu'$-walk up to time instant  $n$.

\begin{lemma}\label{closeuniform}

Given a $c_{n}\times c_{n}$ box $\alpha$, denote by $B_{n}^{\alpha}$
the set of vertices in $\alpha$ which did not have any good multiplication
up to time $n$. For any $\epsilon>0$, there exists constants $\lambda_{1},\lambda_{2}>0$
depending only on $a=\min\{\mu(id),\mu(\delta)\}$ and $\epsilon$,
such that 
\[
\mathbb{P}\left(\left\{ |B_{n}^{\alpha}|\ge(1+\epsilon)(1-a)^{q}|\alpha|\right\} \cap F_{(1-a-\epsilon)n,q}^{\alpha}\right)\le e^{-\lambda_{1}n}+e^{-\lambda_{2}|\alpha|}.
\]

\end{lemma}

\begin{proof}

Let $(N_{i})_{i=1}^{\infty}$ be an i.i.d. sequence of random variables
with distribution $\mathbb{P}(N_{i}=0)=1-a$, $\mathbb{P}(N_{i}=k)=a^{k}(1-a)$
for $k\in\mathbb{N}$, independent of the $\mu'$-random walk $\left(W_{t}'\right)$.
Write $S_{n,\epsilon}=N_{1}+\ldots+N_{(1-a-\epsilon)n}+(1-\alpha-\epsilon)n$.
Denote by $E_{n}$ the event that $S_{n,\epsilon}>n$. By the Chernoff
bound we have that $\mathbb{P}(E_{n})\le e^{-\lambda_{1}n}$ for some constant $\lambda_1$
that only depends on $a$. 

If the box $\alpha$ is $q$-fully visited by the $\mu'$-walk by
the time $(1-\alpha-\epsilon)$, then at time $S_{n,\epsilon}$ for
the walk $\left(\tilde{W}_{t}\right)$, the probability that a vertex
$x$ has not received a good multiplication is bounded from above
by 
\[
\mathbb{P}\left(N_{1}+\ldots+N_{q}=0\right)=(1-a)^{q}.
\]
Therefore by the Chernoff bound, we have 
\[
\mathbb{P}\left(\left\{ B_{S_{n,\epsilon}}^{\alpha}\ge(1+\epsilon)(1-a)^{q}|\alpha|\right\} \cap F_{(1-a-\epsilon)n,q}^{\alpha}\right)\le e^{-\lambda_{2}|\alpha|}.
\]

\end{proof}

Given two configurations $\phi_{1},\phi_{2}\in\{0,1\}^{\alpha}$ in the box $\alpha$,
denote by $d_{H}(\phi_{1},\phi_{2})$ the Hamming distance between
them, that is, $d_{H}\left(\phi_{1},\phi_{2}\right)=\left|\left\{ x\in\alpha:\ \phi_{1}(x)\neq\phi_{2}(x)\right\} \right|$.

Note that if a site $x\in\alpha$ has received a good update before time
$n$, then the lamp configuration at $x$ is uniform on $\{0,1\}$
at time $n$; moreover the configurations at such sites are independent.
The coupling is defined as: $\Psi_{n}$ takes the same value as $\Phi_n$
at a point in $\alpha$ which receives a good update before $n$;
and at a point
that does not receive a good update before time $n$, the value of
$\Psi_{n}$ is independent uniform on $\{0,1\}$. Then by definitions
we have $d_{H}\left(\Psi_{n},\Phi_{n}|_{\alpha}\right)\le|B_{n}^{\alpha}|$,
where $B_n^{\alpha}$ is the set of vertices in $\alpha$ which has not received a good update 
up to time $n$ as in Lemma \ref{closeuniform}. 

To control TSP of the configuration, we need the following classical
bound due to Bartholdi and Platzman \cite{BP82} (which they prove
by space-filling curves heuristics; similar multiscale bound for TSP
works for general graphs other than $\mathbb{Z}^{2}$, see e.g., the
inequality (17) in \cite{NaorPeres}).

\begin{lemma}[Bartholdi-Platzman \cite{BP82}]\label{multiscale}

Let $A$ be a set of $N$ points in a square of side length $M$ in
$\mathbb{Z}^{2}$. Then ${\rm \ell_{\mathrm{TS}}}(A)\le2\sqrt{NM^{2}}$.

\end{lemma}

The rest of this section is devoted to the proof of Theorem \ref{main}. 
\begin{proof}[Proof of Theorem \ref{main}]

Recall that we write $a=\min\left\{ \mu(id),\mu(\delta)\right\} $
and $\mu=au+(1-a)\mu'$, where $u$ is the uniform measure on $\{id,\delta\}$.
For any $\epsilon>0$, take $q$ sufficiently large such that $10(1-a)^{q}<\epsilon$.

Divide $\mathbb{Z}^{2}$ by boxes of side length $c_{n}$. We first
apply the result from \cite{DGK} that almost surely $\left(R_{n}\right)$
forms a F{\o}lner sequence, and Flatto's result in \cite{flatto}
to show that among the visited boxes, the majority of them are $q$-fully
visited by the $\mu'$-walk. More precisely, by \cite{flatto}, we have that
\[
\lim_{n\to\infty}\frac{\left|T_{n}^{q-1}\right|}{n/\log^{2}n}=c_{q-1}\mbox{ a.s.,}
\]
where $T_n^{p}$ is the number of points visited at least once and at most $p$ times by the random walk up to time $n$.
Note that the results in \cite{flatto} is proven for simple random
walk on $\mathbb{Z}^{2}$, but the proof applies to any centered random
walk on $\mathbb{Z}^{2}$ with finite $(2+\epsilon)$-moment, see
the Remark after \cite[Theorem 3.1]{flatto}. If a visited box $\alpha$
is not $q$-fully visited, then it must contain a point of $T_{n}^{q-1}$.
Therefore the number of visited boxes that are not $q$-fully visited
is bounded by $T_{n}^{q-1}$+$\left|\partial R_{n}\right|$.

In a box $\alpha$ that is $q$-fully visited by the $\mu'$-walk
up to time $n$, consider the restriction of lamp configuration $\Phi_{n}$
to $\alpha$. Recall that $B_{n}^{\alpha}$ denotes the set of vertices
in $\alpha$ without good update up to time $n$. As in Lemma \ref{closeuniform},
the configuration $\Phi_{n}|_{\alpha}$ can be written as $U_{\alpha}+\Delta_{n}^{\alpha}$,
where the distribution of $U_{\alpha}$ is uniform on $\{0,1\}^{\alpha}$
and $\Psi_{n}$ is a random configuration with support on the set
$B_{\alpha}$. It follows that to write the configuration $\Phi_{n}|_{\alpha}$
one can first write $U_{\alpha}$, then correct it by $\Delta_{n}^{\alpha}$
which is supported on $B_{\alpha}^{n}$, a subset of small size.

To prove the statement, we show that typically $|X_{n}|_{S}$ is approximated
by the contribution from $q$-fully visited box $\sum_{\alpha\mbox{ is }q\mbox{-full}}\ell_{{\rm TS}}^{(G,S)}(U_{\alpha})$
with error bounded by $\epsilon n/\log n$.

\textbf{The upper bound for $\left|X_{n}\right|_{S}$:}
we exhibit an $S$-path that connects the identity to $X_{n}$. Recall that we have divided
$\mathbb{Z}^2$ into boxes of side length $c_n$.
Consider all visited boxes and denote by $\mathcal{P}_n$ the union of the perimeters of the visited boxes.
Because the step distribution $\mu$ allows long range jumps, the union of the perimeters
of the visited boxes are not necessarily connected. We add paths to
make the union of the perimeters connected. Note that it is sufficient
to add connecting paths with total length bounded by the sum of random
walk jumps longer than $c_{n}$. Indeed, we may enumerate the visited
boxes according to the order they are visited as $\alpha_{1},\alpha_{2},...$:
$\alpha_{i}$ is the first box visited by the $\bar{\mu}$-random
walk after $\alpha_{1},...,\alpha_{i-1}$ that is different from the
previous boxes. If the box $\alpha_{k}$ is not connected to $\alpha_{1}\cup...\cup\alpha_{k-1}$,
then we consider the random walk step that first enters $\alpha_{k}$
and take a shortest path connected the starting and ending point of
this step. The union of such additional paths makes the union of the
visited boxes connected and the length of the additional paths is
bounded by 
\[
A_{n}=\sum_{i=1}^{n}|Y{}_{i}|_{S}{\bf 1}_{\{|Y_{i}|_{S}\ge c_{n}\}}.
\]

Take first an $S$-path that starts at identity and covers the perimeters of each visited box, where connecting paths are added according to the long jumps of the random walk as described above. 
To connect $id$ to $X_{n}$ by an $S$-path,
along this path in each box $\alpha_{i}$, we take the TSP of the configuration
$\Phi_{n}|_{\alpha_{i}}$ which starts and ends on the perimeter of
$\alpha_{i}$. 
Then we can bound the length of the path by 
\begin{align}
|X_{n}|_{S} & \le\sum_{\alpha\ \mbox{is }q\mbox{-full}}\left(\ell_{{\rm TS}}^{(G,S)}(U_{\alpha})+\ell_{{\rm TS}}^{(G,S)}(B_{n}^{\alpha})\right)\nonumber \\
 & +C_0[c_{n}^{2}\left|\left\{ \alpha:\alpha\mbox{ is visited but not }q\mbox{-full}\right\} \right|+|\mathcal{P}_{n}|]+A_{n},\label{eq:up}
\end{align}
where $C_0$ is a constant that only depends on the generating set $S$.
Now we choose appropriate $c_{n}$ such that the main contribution
to the quantity on the righthand side of (\ref{eq:up}) is from $\sum_{\alpha\ \mbox{is }q\mbox{-full}}\ell_{{\rm TS}}^{(G,S)}(U_{\alpha})$.
To this end, we need the expectation of $A_{n}$ to satisfy 
\begin{equation}
\mathbb{E}[A_{n}]=n\sum_{|g|>c_{n}}|g|\mu(g)\ll n/\log n;\label{eq:momenttail}
\end{equation}
and at the same time to be not too large so that the main contribution
still comes from fully visited boxes. More precisely, we need $c_{n}$
to satisfy both the equation (\ref{eq:momenttail}) above and 
\[
c_{n}^{2}\left(\left|R_{n}^{q-1}\right|+\left|\partial R_{n}\right|\right)\ll n/\log n.
\]
Since $\mu$ is assumed to have finite second moment, we have that
$\delta_{r}=\sum_{|g|\ge r}|g|^{2}\mu(g)\to0$ as $r\to\infty$. Thus
we can choose $c_{n}$ to be 
\begin{equation}
c_{n}:=\sqrt{\delta_{\log n}\log n}.\label{eq:cn}
\end{equation}

\textbf{The lower bound for $\left|X_{n}\right|_{S}$:} given an $S$-path
connecting $id$ to $X_{n}$, we consider its restriction to each
$q$-full box $\alpha$ and use the uncrossing lemma \ref{uncross2}
to show that the restriction can be modified to provide paths that
write the configuration $U_{\alpha}$ in $\alpha$. The restriction
of an $S$-path to $\alpha$ means the steps of the path with at least
one of whose starting and ending points in $\alpha$. When the path
is restricted to a box $\alpha$, we obtain a collection $\mathcal{A}_{\alpha}$
of paths, each with projection to $\mathbb{Z}^{2}$ starting and ending
near the boundary of $\alpha$. More precisely, there is a constant
$K>0$ which only depends on the generating set $S$, such that the
starting and ending points of the paths in $\mathcal{A}_{\alpha}$
are within distance $K$ to the boundary of $\alpha$. Reversing the
directions of some paths if necessary, by concatenating paths we may
assume that in the collection $\mathcal{A}_{\alpha}$, two distinct
paths do not share common starting/ending points. In particular, the
number of paths in the collection is bounded by $K'c_{n}$, where
$K'$ is a constant only depending on $K$. For each path $P$ with
starting point $A$ and ending point $B$, let $A'$ ($B'$ resp.)
be the point on the boundary of $\alpha$ closest to $A$ ($B$ resp.).
Replace $P$ by the path $P'$ which is the concatenation of the shortest
$S$-path from $\left(A',{\bf 0}\right)$ to $\left(A,{\bf 0}\right)$,
the path $P$ and the shortest path $\left(B,\tau_{P}\right)$ to
$\left(B',\tau_{p}\right)$. Note that $|P'|\le|P|+2K$. Denote by
$\mathcal{A}_{\alpha}'$ the collection of modified paths obtained
from $\mathcal{A}_{\alpha}$. The purpose of making such modifications
is to assure that paths in the collection $\mathcal{A}_{\alpha}'$
start and terminate on the boundary of $\alpha$. Applying the uncrossing
Lemma \ref{uncross2} to the collection $\mathcal{A}_{\alpha}'$,
we obtain a new collection $\mathcal{B}_{\alpha}$ of paths without
essential crossings, such that $\mathcal{B}_{\alpha}$ produces the
lamp configuration as $\mathcal{A}_{\alpha}$; and moreover the total
length of paths in $\mathcal{B}_{\alpha}$ is bounded by 
\[
\sum_{Q\in\mathcal{B}_{\alpha}}|Q|\le\sum_{P'\in\mathcal{A}_{\alpha}'}|P'|+Cc_{n}\le\sum_{P\in\mathcal{A}_{\alpha}}|P|+C'c_{n},
\]
where $C$ and $C'$ are constants only depending on the generating
set $S$. Reversing the direction of some paths in $\mathcal{B}_{\alpha}$
if necessary (this is allowed because the lamp group is $\mathbb{Z}/2\mathbb{Z}$,
see Remark \ref{reverse}), we can connect the paths in $\mathcal{B}_{\alpha}$
along the perimeter of $\alpha$. This way we produce an $S$-path
which produces the configuration $\Phi_{n}|_{\alpha}$, and the length
of this path is bounded from above by $\sum_{P\in\mathcal{A}_{\alpha}}|P|+C''c_{n}$.
To produce the configuration $U_{\alpha}$, we continue with a path
which writes the difference $\Delta_{n}^{\alpha}$. Summing over the
$q$-full boxes, we obtain the lower bound 
\begin{equation}
|X_{n}|_{S}+\sum_{\alpha\ \mbox{is }q\mbox{-full}}\left(C''c_{n}+\ell_{{\rm TS}}^{(G,S)}(B_{n}^{\alpha})\right)\ge\sum_{\alpha\ \mbox{is }q\mbox{-full}}\ell_{{\rm TS}}^{(G,S)}(U_{\alpha}).\label{eq:lo}
\end{equation}

\textbf{End of the proof}: with the bounds (\ref{eq:up}) and (\ref{eq:lo}),
the choice of $c_{n}$ as in (\ref{eq:cn}), it suffices to show that
a.s., 
\begin{align*}
\lim_{n\to\infty}\frac{1}{n/\log n}\sum_{\alpha\ \mbox{is }q\mbox{-full}}\ell_{{\rm TS}}^{(G,S)}(U_{\alpha}) & =c\\
\lim_{n\to\infty}\frac{1}{n/\log n}\sum_{\alpha\ \mbox{is }q\mbox{-full}}\ell_{{\rm TS}}^{(G,S)}(B_{n}^{\alpha}) & =0.
\end{align*}
By the law of large numbers of arrays of independent random variables
to $\left\{ \ell_{{\rm TS}}^{(G,S)}(U_{\alpha})\right\} $ and $\left\{ \ell_{{\rm TS}}^{(G,S)}\left(B_{n}^{\alpha}\right)\right\} $,
the first limit follows from Lemma \ref{dilutedS1} and the second
limit follows from Lemma \ref{closeuniform} and Lemma \ref{multiscale}.

\end{proof}

\begin{remark}
The only place where we use the $(2+\epsilon)$-moment condition in the proof is in Flatto's result \cite{flatto}.
Recall that $T_n^{p}$ denotes the number of points visited at least once and at most $p$ times by the random walk on $\mathbb{Z}^2$ up to time $n$. We have shown that if we have a centered non-degenerate probability
measure $\mu$ of finite second moment on $G$, such that its projection to $\mathbb{Z}^2$ satisfies that for any $q\ge 2$,
\[
\limsup_{n\to\infty}\frac{\left|T_{n}^{q-1}\right|}{n/\log^{2}n}<\infty\mbox{ almost surely,}
\]
then the statement of Theorem \ref{main} would be true for $\mu$.
\end{remark}

\begin{remark}

For wreath product $\mathbb{Z}^{2}\wr F$ where $F$ is an arbitrary
non-trivial finite group, we say that a generating set $S$ is complete,
if for any $s=(z,f)\in S$, the element $(-z,\tau_{-z}f)$ is also
in the generating set $S$. When $F=\mathbb{Z}/2\mathbb{Z}$, the
completeness of a generating set $S$ is equivalent to symmetry of
$S$. The proof of Theorem \ref{main} extends to symmetric non-degenerate
random walks of finite $(2+\epsilon)$-moment on $\mathbb{Z}^{2}\wr F$,
when the length function $l_{S}$ is associated with a complete generating
set $S$.

\end{remark}

\begin{remark}

When the generating set $S$ of the wreath product $\mathbb{Z}^{2}\wr F$
is not complete, there are configurations where one cannot replace
an $S$-path by an $S$-path of comparable length reversing the starting
and ending points on $\mathbb{Z}^{2}$. For example, take $F=\mathbb{Z}/3\mathbb{Z}$
and the generating set $S$ which consists of $\left(e_{1},{\bf 0}\right),\left(e_{2},{\bf 0}\right)$,
$\left(e_{2},\delta_{0}^{1}\right)$ and their inverses. In a box
of side length $m$, consider the lamp configuration which is $1$
on one vertical line and $0$ everywhere else, see Figure \ref{notcomplete}.
Then for such a configuration, the optimal $S$-path whose projection
to $\mathbb{Z}^{2}$ travels from bottom to top (red to blue in the
picture) is of length $m$; while the optimal $S$-path which travels
from bottom to top is of length at least $2m$.

\begin{figure}
\includegraphics[scale=1.4]{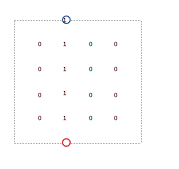} \caption{The picture illustrates in
$\mathbb{Z}^2\wr{\mathbb{Z}/3\mathbb{Z}}$ a configuration on a square that the shortest
length $S$-path traveling from bottom to top is much shorter than
an $S$-path from top to bottom.}
\label{notcomplete} 
\end{figure}
\end{remark}

\section{Wreath products with infinite lamp groups}

In this section we discuss some examples of wreath products with infinite
lamp groups. Consider the wreath product $B\wr H=\left(\oplus_{x\in B}H\right)\rtimes B$,
where both $B$ and $H$ are infinite. Let $P$ be a generating set
of $B$ and $T$ a generating set of $H$. We equip $B\wr H$ with
the generating set $S=\left\{ \left(s,id\right),s\in P\right\} \cup\left\{ \left(id_{B},\delta_{id_{B}}^{t}\right),t\in T\right\} $,
where $\delta_{id_{B}}^{t}:B\to H$ is the function that takes value
$t$ at $id_{B}$ and $id_{H}$ otherwise. Denote by $\left|\cdot\right|_{S}$
the word length on $B\wr H$ with respect to the generating set $S$.

Assume that $B$ is $\mathbb{Z}$ or $\mathbb{Z}^{2}$ and $H$ is
a finitely generated infinite group. Take a symmetric nondegenerate
probability measure $\nu$ of finite support on $B$ and a symmetric
nondegenerate probability measure $\eta$ on $H$. On the wreath product
$B\wr H$, take the switch-walk-switch distribution $\mu=\eta\ast\nu\ast\eta.$
In addition we assume that the drift of the $\eta$-random walk on
$H$ with respect to word length $\left|\cdot\right|_{T}$ has finite
second moment and satisfies 
\begin{equation}
\lim_{t\to\infty}\frac{L_{\eta}(t)}{t^{\alpha}}=c\label{eq:eta-c}
\end{equation}
for some constants $\alpha\in[1/2,1]$ and $c\in(0,\infty)$.

With respect to the generating set $S$ we have for a group element
$\left(z,f\right)\in B\wr H$, 
\begin{equation}
\sum_{x\in{\rm supp}(f)}|f(x)|_{T}\le|(z,f)|_{S}\le\sum_{x\in{\rm supp}(f)}|f(x)|_{T}+2|{\rm supp}f|+|z|_{P}.\label{eq:a0}
\end{equation}
Let $\left(X_{n}\right)_{n=0}^{\infty}$ be a random walk on $B\wr H$
with step distribution $\mu$ and write $X_{n}=(\bar{X}_{n},\Phi_{n})$.
Since the lamp group $H$ is assumed to be infinite and the rate of
escape of the $\eta$-random walk on $H$ satisfies (\ref{eq:a0}),
the length estimate (\ref{eq:a0}) shows that the main contribution
to $\left|X_{n}\right|_{S}$ is from $\sum_{x}|\Phi_{n}(x)|_{T}$.
Denote by $\left(Y_{t}^{x}\right)_{t=0}^{\infty}$ a collection of
independent $\eta$-random walks, indexed by $x\in B.$ Let $l(n,x)$
be the number of visits to $x\in B$ by the projected random walk
$\bar{X}=(\bar{X}_{n})_{n=0}^{\infty}$ on the base. Then $\left(\Phi_{n}(x)\right)_{x\in B}$
has the same distribution as $\left(Y_{l(n,x)}^{x}\right)_{x\in B}$.
In particular, conditioned on the local times $\left(l(n,x)\right)_{x\in B}$,
the random variables $\Phi_{n}(x)$, $x\in B$, are independent. Write
$L_{\eta}(t)=\mathbb{E}\left[|Y_{t}^{x}|_{T}\right]$ and $V_{\eta}(t)={\rm Var}\left(|Y_{t}^{x}|_{T}\right)$.
Then by the law of large numbers for independent random variables,
we have almost surely 
\begin{equation}
\lim_{n\to\infty}\frac{\sum_{x\in R_{n}}\left|Y_{l(n,x)}^{x}\right|_{T}}{\sum_{x\in R_{n}}l(n,x)^{\alpha}}\to c,\label{eq:as1}
\end{equation}
where $c$ is the limit of $L_{\eta}(t)/t^{\alpha}$, as assumed in
(\ref{eq:eta-c}). The problem is reduced to the expression $\sum_{x\in R_{n}}l(n,x)^{\alpha}$
of local times on the base.

For example, when $H=\mathbb{Z}$ and $\eta$ is a symmetric nondegenerate
probability measure on $\mathbb{Z}$ with finite second moment, then
the assumption (\ref{eq:a0}) is satisfied with $\alpha=1/2$ and
the variance satisfies $V_{\eta}(t)\simeq t$.

Now we consider an example with $\mathbb{Z}$ as a base group.

\begin{example}\label{1-dim infinite lamp}

Let $B=\mathbb{Z}$ and $H=\mathbb{Z}$. Take the probability measure
$\nu$ on $B$ to be uniform on $\{\pm1\}$ and take $\eta$ on $H$
to be uniform on $\left\{ \pm1\right\} $. On $\mathbb{Z}\wr\mathbb{Z},$
consider the random walk $\left(X_{n}\right)_{n=0}^{\infty}$ with
the standard SWS step distribution $\mu=\eta\ast\nu\ast\eta$. Then
by the invariance principle for local times of simple random walk
on $\mathbb{Z}$ due to Perkins \cite{perkins}, we have that 
\[
\frac{1}{n^{3/4}}\sum_{x\in\mathbb{Z}}l(n,x)^{1/2}\Rightarrow\int_{-\infty}^{\infty}\left(L_{1}^{x}\right)^{1/2}dx,
\]
where $L_{t}^{x}$ is the local time of the standard Brownian motion.

\end{example}

Next we return to the case of $\mathbb{Z}^{2}$ as a base group. Compared
to Theorem \ref{main}, it is easier to establish the law of large
numbers for random walks on the wreath product over $\mathbb{Z}^{2}$
with an infinite group as lamps. For a switch-walk-switch random walk
on $\mathbb{Z}^{2}\wr\mathbb{Z}$, the problem is reduced to functionals
of local times on $\mathbb{Z}^{2}$, where a result of \v{C}ern{\'y} \cite{cerny}
can be applied.

\begin{example}\label{2-dim Z-lamp}

Let $B=\mathbb{Z}^{2}$ and $H=\mathbb{Z}$. Let $\left(X_{n}\right)_{n=0}^{\infty}$
be a random walk on $\mathbb{Z}^{2}\wr\mathbb{Z}$ with switch-walk-switch
step distribution $\eta\ast\nu\ast\eta$, where $\eta$ is uniform
on $\{\delta^{-1},\delta\}$ and $\nu$ is uniform on $\{s_{1}^{\pm1},s_{2}^{\pm1}\}$.
Then there exists a constant $c\in(0,\infty)$ 
\[
\lim_{n\to\infty}\frac{|X_{n}|_{S}}{n/\log^{1/2}n}=c\mbox{ a.s.}
\]

\end{example}

\begin{proof}

By \v{C}ern{\'y}'s result \cite{cerny}, we know that there is a constant
$C>0$ such that almost surely 
\begin{equation}
\lim_{n\to\infty}\frac{\sum_{x\in R_{n}}l(n,x)^{1/2}}{n(\log n)^{-1/2}}=C.\label{eq:as2}
\end{equation}
The claim in the example follows from combining (\ref{eq:as1}) and
(\ref{eq:as2}), and the fact that the other two terms on the righthand
side of (\ref{eq:a0}) satisfy 
\begin{align*}
\left|{\rm supp}\Phi_{n}\right| & \le|R_{n}|\mbox{ and }\lim_{n\to\infty}|R_{n}|/n(\log n)^{-1/2}=0\ \mbox{a.s.},\\
 & \lim_{n\to\infty}|\bar{X}_{n}|/n(\log n)^{-1/2}=0\mbox{ a.s.}
\end{align*}

\end{proof}

\section{Controlled F{\o}lner pairs and limiting behavior of $l_{S}(X_{n})/L_{\mu}(n)$\label{sec:questions}}

In general $l_{S}(X_{n})/L_{\mu}(n)$ does not necessarily converge
in distribution. A strong opposite of concentration of $l_{S}(X_{n})$
around its mean is that $l_{S}(X_{n})/L_{\mu}(n)$ converges in distribution,
and the limiting distribution admits a density which is supported
on the whole ray $(0,\infty)$. In this section we consider the class
of groups which admits controlled F{\o}lner pairs.

\subsection{Questions}

F{\o}lner pairs were introduced in Coulhon, Grigoryan and Pittet
\cite{CGP} to produce lower bounds on return probability. We recall
the definition: a sequence $(F'_{n},F_{n})$ of pairs of finite subsets
of $G$ with $F'_{n}\subset F_{n}$ is called a sequence of \textit{F{\o}lner
pairs} adapted to an increasing function $\mathcal{V}(n)$ if there
is a constant $C<\infty$ and such that for all $n$, we have 
\begin{description}
\item [{(i)}] $\#F_{n}\le C\#F'_{n}$, 
\item [{(ii)}] $d(F'_{n},G\setminus F_{n})\ge n$, 
\item [{(iii)}] $\#F_{n}\le\mathcal{V}(Cn)$. 
\end{description}
F{\o}lner pairs provide lower bound for return probability $\mu^{(2n)}(e)$
for a symmetric probability measure $\mu$ of finite support on $G$,
see \cite{CGP}. The definition of controlled F{\o}lner pairs was
introduced by Tessera in \cite{tessera}. We say $(F'_{n},F_{n})$
with $F'_{n}\subset F_{n}$ is a sequence of \textit{controlled F{\o}lner
pairs} if they satisfy (i) and (ii) as above, and 
\begin{description}
\item [{(iii')}] $F_{n}\subset B(e,Cn)$. 
\end{description}
Note that it follows from definition that controlled F{\o}lner pairs
are adapted to the volume growth function of the group.

We ask the following question and in this section provide some evidence
supporting a positive answer.

\begin{question}\label{controlled-density}

Is it true that admitting controlled F{\o}lner pairs implies convergence
of $l_{S}(X_{n})/L_{\mu}(n)$ in distribution and the limiting distribution
admits a positive density on $(0,\infty)$?

\end{question}

In the context of the question, we mention the following asymptotic
properties of finitely generated groups. 
\begin{enumerate}
\item The F{\o}lner function is equivalent to the growth function $v_{S}(n)$. 
\item The group admits a sequence of F{\o}lner pairs $\left(F_{n}',F_{n}\right)$
with volume of $F_{n}$ equivalent to the growth function $v_{S}(n)$. 
\item The group admits controlled F{\o}lner pairs. 
\item The return probability $\mathbb{P}(X_{2n}=e)$ of simple random walk
to the origin is equivalent to the solution $\psi(n)$ of 
\[
\int_{1}^{1/\psi(t)}\frac{1}{sv_{S}^{-1}(s)^{2}}ds=t.
\]
\item Simple random walk on the group has diffusive rate of escape, that
is, $L_{\mu}(n)\simeq\sqrt{n}$. 
\item Simple random walk on the group is cautious (we will recall the definition
of cautiousness in Subsection \ref{subsec:controlled example}). 
\item The rate of escape is diffusive and the distribution of $|X_{n}|_{S}/L_{\mu}(n)$
converges in distribution to a limiting density whose support is $(0,\infty)$. 
\end{enumerate}
So far there is no known example that satisfies one of the conditions
in the list but not all the others. All known examples satisfying
these conditions are elementary amenable, in particular, there is
no known example of group of intermediate growth with properties in
the list. One can also ask when the rate of escape is diffusive, whether
the entropy of the random walk is equivalent to $\log v_{S}\left(\sqrt{n}\right)$.
Another question one can ask in the setting of condition 7 is whether
$|X_{n}|_{S}/L_{\mu}(n)$ exhibits Gaussian tail decay, that is, there
exists a constant $c>0$ such that 
\[
\mathbb{P}\left(|X_{n}|_{S}/L_{\mu}(n)\ge x\right)\le e^{-cx^{2}}\mbox{ for }x>1.
\]

We remark that properties in the list 1-6 describe some asymptotic
properties of the group that are minimal possible. Recall that by
\cite{LeePeres}, the rate of escape of simple random walk on an infinite
group has a diffusive lower bound. Property $5$ in the list describes
that the rate of escape is minimal possible in the sense that it is
equivalent to the universal lower bound. By the Coulhon-Saloff-Coste
inequality, the F{\o}lner function is at least the volume growth.
For example, on a group of exponential volume growth, the Coulhon-Saloff-Coste
inequality implies that the F{\o}lner function is at least exponential
and the return probability satisfies $\mathbb{P}(X_{2n}=e)\preceq e^{-n^{1/3}}$.
Property 1 and 4 describe that the F{\o}lner function (return probability
of simple random walk resp.) is minimal (maximal resp.) possible given
the volume growth. Properties 2 and 3 are strengthenings of Property
1, concerning not only the cardinalities of F{\o}lner sets, but also
the shape and geometry of them.

One class of groups where the positive answer to Question \ref{controlled-density}
is known consists of groups of polynomial growth. Note that groups
of polynomial growth admit controlled F{\o}lner pairs. As already
mentioned in the Introduction, on groups of polynomial growth, one
can apply the local limit theorem in Alexopoulos \cite{alexopoulos}.
More precisely, let $G$ be a group of polynomial growth and $\mu$
be a symmetric probability on $G$ with finite generating support.
Then \cite[Corollary 1.19]{alexopoulos} states that for any $\epsilon\in(0,1)$,
there is a constant $C>0$ such that 
\[
\left|\mu^{(n)}(x^{-1}y)-p_{n}^{H_{\mu}}\left(x,y\right)\right|\le Cn^{-\frac{1}{2}(D+\epsilon)}\exp\left(-\frac{l_{S}(x^{-1}y)^{2}}{Cn}\right),
\]
where $H_{\mu}$ is the limiting Laplacian on a nilpotent Lie group
$N$ which is referred to as the stratified Lie group associated to
$G$ in \cite{alexopoulos}. The stratified nilpotent group $N$ admits
a dilation $\delta:N\to N$ and the transition probabilities of the
limiting process is scale invariant with respect to the dilation:
$p_{st}^{H_{\mu}}\left(\delta_{\sqrt{s}}x,\delta_{\sqrt{s}}y\right)=p_{t}^{H_{\mu}}(x,y)$.
The local limit theorem implies that $l_{S}(X_{n})/\sqrt{n}$ converges
in distribution and the limiting distribution admits a positive density
on $\mathbb{R}_{+}$.

\subsection{The example of $\mathbb{Z}\wr F$}

In this subsection we explain the case of the lamplighter $\mathbb{Z}\wr F$
over $\mathbb{Z}$ with the lamp group $F$ a finite nontrivial group.
The group $\mathbb{Z}\wr F$ is an example of a group of exponential
growth that admits controlled F{\o}lner pairs. We show that for simple
random walk on $\mathbb{Z}\wr F$, the limiting distribution of $l_{S}(X_{n})/L_{\mu}(n)$
satisfies the claim of Question \ref{controlled-density}. Let $\left(B_{t}\right)_{t\ge0}$
be the standard Brownian motion and denote by $\mathcal{R}_{t}$ the
size of the range of the standard Brownian motion $\left(B_{t}\right)_{t\ge0}$
at time $t$.

\begin{proposition}\label{0-1 lamp}

Let $\mu$ be a non-degenerate symmetric measure of finite second
moment on $\mathbb{Z}\wr F$, $F$ is a finite group, $F\neq\left\{ id\right\} $.
Let $S$ be a finite symmetric generating set of $\mathbb{Z}\wr F$.
Then there exists constants $c_{1},c_{2}>0$ such that $l_{S}\left(X_{n}\right)/\sigma\sqrt{n}$
converges in distribution to $c_{1}\left(\mathcal{R}_{1}-|B_{1}|\right)+c_{2}|B_{1}|$,
where $\sigma$ is the variance of $\bar{X}_{1}$.

\end{proposition}

\begin{proof}

We first explain a particular case which is easier. Consider the SWS
generating set $S_{0}$ which consists of elements of the form $\left(e_{1},\delta_{0}^{\gamma}+\delta_{1}^{\gamma'}\right)$,
where $\gamma,\gamma'\in F$, and their inverses. Take the random
walk step distribution $\mu_{0}$ to be uniform on $S_{0}$. Recall
that for an element $(x,f)\in(\mathbb{Z}/2\mathbb{Z})\wr\mathbb{Z}$,
for the SWS generating set $S_{0}$, we have 
\[
l_{S_{0}}\left(\left(x,f\right)\right)=2\max\left\{ {\rm supp}f\right\} -2\min\left\{ {\rm supp}f\right\} -|x|,
\]
if $x\in\left[\min\left\{ {\rm supp}f\right\} ,{\rm max}\left\{ {\rm supp}f\right\} \right]$.
Denote by $R_{n}$ the range of the projected $\mu_{0}$-random walk
on $\mathbb{Z}$ at time $n$. Then there is a constant $C>0$, almost
surely for all $n$ sufficiently large, 
\begin{equation}
\left|l_{S}\left(X_{n}\right)-\left(2|R_{n}|-|\bar{X}_{n}|\right)\right|\le C\log n.\label{eq:LR}
\end{equation}
In \cite{jainpruitt}, a central limit theorem for $|R_{n}|$ is shown:
$|R_{n}|/\sqrt{n}$ converges to $\mathcal{R}_{1}$ in distribution.
This follows from Donsker's invariance principle, and similarly one
can show that $\left(2|R_{n}|-|\bar{X}_{n}|\right)/\sqrt{n}$ converges
to $2\mathcal{R}_{1}-|B_{1}|$ in distribution. Denote by $f_{\infty}$
the density of the random variable $2\mathcal{R}_{1}-|B_{1}|$. Thus
by (\ref{eq:LR}), we have that the limiting density of $l(X_{n})/\sqrt{n}$
is $f_{\infty}$.

Next we consider the general case. We argue in a way similar to the
proof of Theorem \ref{main}, except that in the one-dimensional case
the uncrossing lemma is not needed. Note the following consequence
of sub-additivity.

\begin{lemma}\label{sub-additive-1d}

Let $\mathcal{\Lambda}_{n}$ be a configuration uniform on $\oplus_{z\in\left\{ 1,\ldots,n\right\} }(F)_{z}$.
Denote by ${\rm TSP}_{S}\left(\Lambda_{n};x,y\right)$ the shortest
$S$-path that writes $\Lambda_{n}$ whose projection in $\mathbb{Z}$
starts at $x$ and ends at $y$. Then there are constants $c_{1},c_{2}>0$
such that almost surely 
\begin{align*}
\lim_{n\to\infty}\frac{1}{n}{\rm TSP}_{S}\left(\Lambda_{n};1,1\right) & =c_{1},\\
\lim_{n\to\infty}\frac{1}{n}{\rm TSP}_{S}\left(\Lambda_{n};1,n\right) & =c_{2}.
\end{align*}

\end{lemma}

\begin{proof}{[}Proof of Lemma \ref{sub-additive-1d}{]}

Let $\Lambda$ be the i.i.d. uniform configuration on $F^{\mathbb{Z}}.$
Given two integers $x\le y$, denote by $\Lambda_{[x,y]}$ the restriction
of $\Lambda$ to the interval $\{x,\ldots,y\}$. Then there is a constant
$C$ which depends only on $S$ such that for any $x\le y\le z$,
we have 
\begin{equation}
{\rm TSP}_{S}\left(\Lambda_{[x,z]};x,x\right)\le{\rm TSP}_{S}\left(\Lambda_{[x,y]};x,x\right)+{\rm TSP}_{S}\left(\Lambda_{[y,z]};y,y\right)+C.\label{eq:1d-sub1}
\end{equation}
Indeed if we have paths $P_{1}$ of optimal length starting and ending
at $x$ producing $\Lambda_{[x,y]}$, and $P_{2}$ of optimal length
starting and ending at y producing $\Lambda_{[y,z]}$, then we may
obtain the following path which writes $\Lambda_{[x,z]}$. Follow
$P_{1}$, which starts at $x$, until the first time it exits the
interval $\left[x,y\right]$, call this segment of $P_{1}$ as $P_{1}'$.
Denote by $y'$ the end point of $P_{1}'$, it is within bounded distance
$C_{0}$ to $y$. Then to write $\Lambda_{[x,z]}$, we first follow
$P_{1}'$, then move from $y'$ to $y$ without changing the configuration,
then follow $P_{2}$, move back from $y$ to $y'$, and finally follow
the rest of $P_{1}$. The inequality (\ref{eq:1d-sub1}) follows.

By concatenating $S$-paths, we have that 
\begin{equation}
{\rm TSP}_{S}\left(\Lambda_{[x,z]};x,z\right)\le{\rm TSP}_{S}\left(\Lambda_{[x,y]};x,y\right)+{\rm TSP}_{S}\left(\Lambda_{[y,z]};y,z\right)+C.\label{eq:1d-sub2}
\end{equation}
Given (\ref{eq:1d-sub1}) and (\ref{eq:1d-sub2}), the statement follows
from Kingman's subadditive ergodic theorem.

\end{proof}

Now we return to the proof of Proposition \ref{0-1 lamp}. We show
that there exist constants $c_{1},c_{2}>0$ (these will be the constants
in Lemma \ref{sub-additive-1d}) such that almost surely 
\begin{equation}
\lim_{n\to\infty}\frac{l_{S}\left(X_{n}\right)}{c_{1}(R_{n}-|\bar{X}_{n}|)+c_{2}|\bar{X}_{n}|}=1,\label{eq:1d-conv}
\end{equation}
where $\bar{X}_{n}$ is the projection of $X_{n}$ to $\mathbb{Z}$
and $\left|\bar{X}_{n}\right|$ is its absolute value.

The proof of (\ref{eq:1d-conv}) is similar to that of Theorem \ref{main},
except that the $1$-dimensional case is easier and does not require
the uncrossing lemma (while the uncrossing lemma is essential in the
$2$-dimensional case in Theorem \ref{main}). Given any constant
$\epsilon>0$, at time $n$ choose $c_{n}$ and divide $\mathbb{Z}$
into intervals of length $c_{n}$. Note that by \cite[Theorem 21]{DGK}
(see also the remark after it), we have that for any $\delta>0$,
almost surely 
\[
\frac{\left|\partial R_{n}\right|}{|R_{n}|}=o\left(n^{-\frac{1}{2}+\delta}\right).
\]
We proceed in the same way as in the proof of Theorem \ref{main}.
In the upper bound direction, to obtain an $S$-path we concatenate
paths lying inside each interval and connect disjoint intervals to
the rest of the intervals if necessary. In this way we obtain an upper
bound for $l_{S}(X_{n})$ analogous to (\ref{eq:up}). In the lower
bound direction, we start with an optimal $S$-path connecting id
to $X_{n}$, then up to bounded error, in each interval it induces
a TSP path either of the first or second kind, see Figure \ref{finitelamps}.
In this way we obtain a lower bound analogous to (\ref{eq:lo}). The
convergence to the limit as in (\ref{sub-additive-1d}) is proved
in the same way as in Theorem \ref{main}.

\begin{figure}
\includegraphics[scale=0.5]{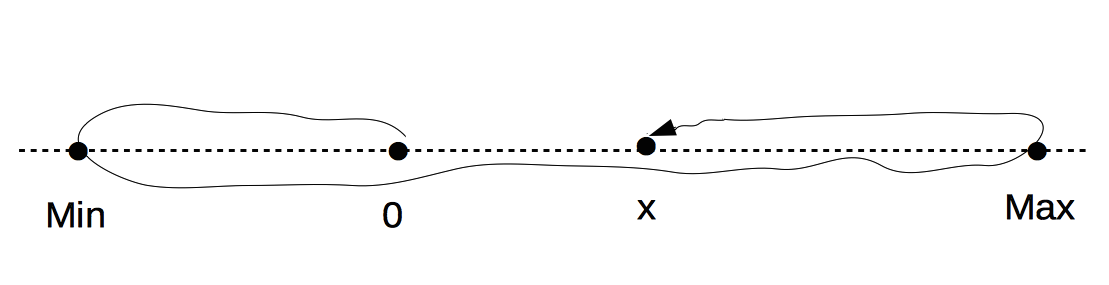} \caption{Given a prescribed configuration on $[\min,\max]$ and end point $x_{n}$,
on the subintervals $[min,0]$ and $[x,max]$ we have TSP problem
of the first kind; on the interval $[0,x]$ we have TSP problem of
the second kind. This results in the two constants $c_{1},c_{2}$.}
\label{finitelamps} 
\end{figure}
Similar to the SWS random walk case, as a consequence of Donsker's
invariance principle, the expression $c_{1}(R_{n}-|\bar{X}_{n}|)+c_{2}|\bar{X}_{n}|$
normalized by $\sigma\sqrt{n}$, converges in distribution to $c_{1}\left(\mathcal{R}_{1}-|B_{1}|\right)+c_{2}|B_{1}|$.
Combined with (\ref{eq:1d-conv}), we conclude that $l_{S}(X_{n})/\sigma\sqrt{n}$
converges in distribution to $c_{1}\left(\mathcal{R}_{1}-|B_{1}|\right)+c_{2}|B_{1}|$.

\end{proof}

\subsection{A partial result in the general case of groups admitting controlled
F{\o}lner pairs\label{subsec:controlled example}}

In this subsection we discuss some general properties of groups admitting
controlled F{\o}lner pairs.

If a group $G$ admits a sequence of controlled F{\o}lner pairs,
then simple random walk on $G$ is diffusive, that is, $L_{\mu}(n)\simeq\sqrt{n}$,
see \cite[Theorem 1.4]{PeresZheng} (the statement is formulated in terms
of $\ell^{2}$-isoperimetric profile, by \cite{CGP} controlled F{\o}lner
pairs imply that the $\ell^{2}$-isoperimetric profile in balls satisfy
$\lambda\left(B(id,r)\right)\lesssim r^{-2}$). As introduced in \cite{erschlerozawa},
the random walk $(X_{n})_{n=0}^{\infty}$ is \emph{cautious} if for
any constant $c>0$, there exists a constant $\delta=\delta(c)>0$,
such that $\mathbb{P}(l_{S}(X_{n})\le c\sqrt{n})\ge\delta(c)$. In
particular, if the random walk $(X_{n})_{n=0}^{\infty}$ is \emph{cautious},
then $l_{S}(X_{n})/L_{\mu}(n)$ does not converge in probability.
Admitting a sequence of controlled F{\o}lner pairs implies that simple
random walks are cautious, see \cite[Lemma 4.5]{EZ17}.

The following lemma shows that in the context of Question \ref{controlled-density},
if the limiting distribution exists, then it must be supported on
the whole ray $\left(0,\infty\right)$.

\begin{lemma}\label{cautious-density}

Let $G$ be a group where the $\mu$-random walk on $G$ is cautious.
Suppose in addition that $l_{S}(X_{n})/\sqrt{n}$ converges in distribution
to a limiting distribution $\nu$. Then for any $(a,b)\subseteq[0,\infty)$,
we have $\nu\left(\left(a,b\right)\right)>0$.

\end{lemma}

\begin{proof}

The cautiousness condition implies that for any $\varepsilon>0$,
$\nu((0,\varepsilon))>0$. By \cite{erschlerozawa}, since the $\mu$-random
walk $\left(X_{n}\right)_{n=0}^{\infty}$ on $G$ is cautious, $G$
admits a virtual homomorphism onto $\mathbb{Z}$. Denote by $G_{1}$
a finite index subgroup of $G$ with a homomorphism $\pi:G_{1}\to\mathbb{Z}$
that is onto. Consider the induced random walk on $G_{1}$: let $\tau_{0}=0$,
$\tau_{k}=\min\{n>\tau_{k-1}:X_{n}\in G_{1}\}$, and $Y_{k}=X_{\tau_{k}}$.
By the law of large numbers, we have that $\tau_{k}/k\to\mathbb{E}[\tau_{1}]$
almost surely when $k\to\infty$. By the central limit theorem on
$\mathbb{Z}$, we have that $\pi\left(X_{\tau_{k}}\right)/\sqrt{n}$
converges to normal distribution $N(0,\sigma^{2})$ in distribution.
Note that there is a constant $c$, which only depends on $G_{1}$
and the generating set $S$, such that $\left|X_{\tau_{k}}\right|_{S}\ge c\left|\pi\left(X_{\tau_{k}}\right)\right|$.
Therefore for any $x>0$, we have 
\begin{align*}
\lim_{n\to\infty}\mathbb{P}\left(\frac{\left|X_{n}\right|_{S}}{\sqrt{n}}\ge x\right) & =\lim_{k\to\infty}\mathbb{P}\left(\frac{\left|X_{\tau_{k}}\right|_{S}}{\sqrt{k\mathbb{E}\left[\tau_{1}\right]}}\ge x\right)\\
 & \ge\lim_{n\to\infty}\mathbb{P}\left(\frac{c\left|\pi\left(X_{\tau_{k}}\right)\right|}{\sqrt{k\mathbb{E}\left[\tau_{1}\right]}}\ge x\right)\\
 & =2\Phi\left(\frac{\mathbb{E}\left[\tau_{1}\right]^{1/2}}{c\sigma}x\right),
\end{align*}
where $\Phi(x)=\mathbb{P}(Z>x)$, $Z$ the standard normal random
variable. In particular, it follows that for any $x>0$, $\nu\left(\left(x,\infty\right)\right)>0$.

Given $0<a<b<\infty$, choose an interval $(A,B)$ such that $A>b$,
$|B-A|<|b-a|/2$ and $\nu\left(\left(A,B\right)\right)>0$. Such an
interval $(A,B)$ exists because $v((x,\infty))>0$ for all $x>0$.
Next, choose a constant $c\in(0,1)$ and a sufficiently small $\varepsilon>0$
such that 
\[
\left(A\sqrt{c}-\varepsilon,B\sqrt{c}+\varepsilon\right)\subseteq\left(a,b\right).
\]
For example, one can choose $c$ to satisfy $A\sqrt{c}=a+\frac{1}{3}(b-a)$
and take $0<\varepsilon<\frac{1}{3}(b-a)$. By the triangle inequality,
we have for $k=cn$, 
\[
\mathbb{P}\left(\left|X_{n}\right|{}_{S}\in\left(a\sqrt{n},b\sqrt{n}\right)\right)\ge\mathbb{P}\left(\left|X_{k}\right|_{S}\in\left(A\sqrt{k},B\sqrt{k}\right)\right)\cdot\mathbb{P}\left(\left|X_{n-k}\right|_{S}\le\varepsilon\sqrt{n-k}\right).
\]
Then in the limit we have 
\[
v\left(\left(a,b\right)\right)\ge\nu((A,B))\nu((0,\epsilon))>0.
\]

\end{proof}

\begin{remark}

As explained before the statement of Lemma \ref{cautious-density},
admitting controlled F{\o}lner pairs implies that simple random walk
on $G$ is cautious. Thus in the context of Question \ref{controlled-density},
if the limiting distribution exists and admits a density, then the
support of the density must be the whole ray $(0,\infty)$.

\end{remark}

\subsection{Dependence of the limiting distribution on subsequences}

For groups which admit controlled F{\o}lner pairs only along a subsequence,
the situation is more complicated. In this subsection we explain an
example of a group $G$ such that for a simple random walk $\mu$
on $G$, along two difference subsequences, the limiting distributions
of $l_{S}(X_{n})/L_{\mu}(n)$ are different. Denote by $f_{\infty}$
the density function of $c_{1}\left(\mathcal{R}_{1}-|B_{1}|\right)+c_{2}|B_{1}|$
as in Proposition \ref{0-1 lamp}.

\begin{proposition}\label{lacunary}

There exists a lacunary hyperbolic group $G$, which is locally-nilpotent-by-$\mathbb{Z}$,
such that along one subsequence $(n_{i})$, $l_{S}(X_{n_{i}})/L_{\mu}(n_{i})$
converges in distribution to the limiting density $f_{\infty}$; and
along another subsequence $(m_{i}),$ $l_{S}(X_{m_{i}})/L_{\mu}(m_{i})$
converges in probability to constant $1$.

\end{proposition}

\begin{proof}

Consider the locally-nilpotent-by-$\mathbb{Z}$ groups which are constructed
in \cite{gromoventropy}. We recall the construction: take a prime
number $p$ and denote by ${\bf M}$ the free product $\ast_{i\in\mathbb{Z}}\left\langle a_{i}\right\rangle $,
where $a_{i}$ satisfies $a_{i}^{p}=1$, $i\in\mathbb{Z}$. Take a
non-decreasing sequence ${\bf c}=\left(c_{i}\right)_{i=1}^{\infty}$
of positive integers and denote by $A(p,\mathbf{c})$ the quotient
group of $\mathbf{M},$ subject to the relations that for each $n\in\mathbb{N}$,
\[
\left[\ldots\left[a_{i_{0}},a_{i_{1}}\right],\ldots,a_{i_{c_{n}}}\right]=1,\ \mbox{where }\max_{j,k}\left|i_{j}-i_{k}\right|\le n.
\]
Denote by $G(p,{\bf c})$ the semi-product $A(p,{\bf c})\rtimes\left\langle t\right\rangle $,
where $t$ acts by the automorphism $a_{i}\to a_{i+1}$, $i\in\mathbb{Z}$,
on $A(p,{\bf c})$.

In terminology introduced later, with appropriate choice of $\mathbf{c}$,
the group $G(p,{\bf c})$ is lacunary hyperbolic, see \cite[Section 3.5]{olshosinsapir}.
As considered in \cite{EZ17}, one can impose additional relators
in $G(p,{\bf c})$ such that the group is close to the lamplighter
$\mathbb{Z}\wr(\mathbb{Z}/p\mathbb{Z})$ along another subsequence.
Let $\ensuremath{G_{0}=(\mathbb{Z}/p\mathbb{Z})\wr\mathbb{Z}}$ and
$\mathbf{M}=\ast_{i\in\mathbb{Z}}\left\langle b_{i}\right\rangle $.
We define a sequence of quotients of $G(p,{\bf c})$ recursively as
follows, parametrized by the sequence $(\ell_{i},c_{i},k_{i})_{i\in\mathbb{N}}$.

After we have defined $G_{i}$, given the parameter $\ell_{i+1}$,
take the nilpotent subgroup in $G_{i}$ generated by $b_{0},\ldots,b_{2^{\ell_{i+1}}-1}$,
\[
N_{i+1}=\left\langle b_{0},\ldots,b_{2^{\ell_{i+1}}-1}\right\rangle .
\]
Consider the quotient group $\hat{M}_{i+1}$ of $M$ defined by imposing
the relations that for any $j$, the subgroup generated by $\left\langle b_{j},\ldots,b_{j+2^{\ell_{i+1}}-1}\right\rangle $
is isomorphic to $N_{i+1}$ (isomorphism given by shifting indices
by $j$). Let $\hat{\Gamma}_{i+1}=\hat{M}_{i+1}\rtimes\mathbb{Z}$
be the cyclic extension of $\hat{M}_{i+1}$. Then $\hat{\Gamma}_{i+1}$
splits as an HNN-extension of the finite nilpotent group $N_{i+1}$,
therefore $\hat{\Gamma}_{i+1}$ is virtually free.

Given the parameters $c_{i+1},k_{i+1}\in\mathbb{N}$, consider the
quotient group $G_{i+1}=\bar{\Gamma}_{i+1}(c_{i+1},k_{i+1})$ of $\hat{\Gamma}_{i+1}$
subject to additional relations $(\ast)$ 
\begin{align*}
\left[\left[\left[b_{j_{1}},b_{j_{2}}\right],b_{j_{3}}\right],\ldots,b_{j_{m}}\right] & =0\mbox{ for any }m\ge c_{i+1},\\
b_{j} & =b_{j+2^{k_{i+1}}}\mbox{ for all }j\in\mathbb{Z}.
\end{align*}

By construction, if we choose $\ell_{i+1}\gg\max\left\{ \ell_{i},k_{i},c_{i}\right\} $
and $k_{i+1},c_{i+1}\gg\ell_{i+1}$ to be large enough parameters,
then we have that $G_{i+1}$ and $G_{i}$ coincide on the ball of
radius $2^{\ell_{i+1}}$ around the identity element, and moreover,
$G_{i+1}$ and $\hat{\Gamma}_{i+1}$ coincide on the ball of radius
$2^{k_{i+1}}$.

On the group $\mathbf{M}\rtimes\mathbb{Z}$, denote by $\mu$ the
uniform measure on $\left\{ t^{\pm1},b_{0}^{\pm1}\right\} $. In what
follows, on a quotient group of $\mathbf{M}\rtimes\mathbb{Z}$, we
consider the random walk with step distribution the projection of
$\mu$.

We now specify the choice of parameters $(\ell_{i},k_{i},c_{i})$.
Fix a sequence of positive numbers $\left(\epsilon_{i}\right)_{i=1}^{\infty}$
such that $\epsilon_{i}$ converges to $0$ when $i\to\infty$. Suppose
$G_{i}$ is defined and we choose $\ell_{i+1},k_{i+1}$ and $c_{i+1}$.
By its definition $G_{i}$ is a finite extension of $(\mathbb{Z}/p\mathbb{Z})\wr\mathbb{Z}$.
More precisely, let $H_{i}$ be the subgroup of $G_{i}$ generated
by $b_{0},\ldots,b_{2^{k_{i}}-1}$. Note that $H_{i}$ is a finite
nilpotent group. Then because of the relations $(\ast)$ in $\bar{\Gamma}_{i}$,
$G_{i}$ fits into the exact sequence 
\[
1\to[H_{i},H_{i}]\to G_{i}\to(\mathbb{Z}/p\mathbb{Z})\wr\mathbb{Z}\to1.
\]

More precisely, one can recursively choose intervals $I_{k},J_{k}$
in $\mathbb{N}$ with $\max I_{k}<\min J_{k}$ and groups $H_{k},G_{k}$
such that

(i) $H_{k}$ is virtually free (and it only depends on $G_{1},H_{1},\ldots,G_{k-1},H_{k-1},G_{k}$)
and for an integer $r\in J_{k}$, the ball of radius $r$ in $H_{k}$
around $id$ coincides with the ball of the same radius around $id$
in $G_{k}.$

(ii) for an integer $r\in I_{k}$, the ball of radius $r$ in $G_{k}$
is the same as the ball of same radius in $H_{k-1}$; and $G_{k}$
is a finite extension of $\mathbb{Z}\wr(\mathbb{Z}/2\mathbb{Z})$
where the extension only depends on choices up to index $k-1$.

Let $G$ be the direct limit of the sequence of groups $\{G_{k}\}.$
In the construction one can take the intervals $I_{k},J_{k}$ to be
arbitrarily long. By the subadditive theorem, on the virtually free
group $H_{k-1}$, we have $|X_{n}^{(k-1)}|_{S}/L_{(k-1)}(n)\to1$
almost surely as $n\to\infty$. Thus we can choose $I_{k}$ to be
long enough such that almost surely for all $n\ge m(I_{k})$, $\left||X_{n}^{(k-1)}|_{S}/L_{(k-1)}(n)-1\right|<\epsilon_{k}$.
To choose $J_{k}$, we make sure that it is sufficiently long such
that diameter of kernel of $G_{k}\to\mathbb{Z}\wr(\mathbb{Z}/2\mathbb{Z})$
is less than $\epsilon_{k}m(J_{k})$.

It follows that if the intervals $I_{k},J_{k}$ are sufficiently long,
in the direct limit $G$, along one subsequence of time, the random
walk sees the virtually free group; and along another subsequence
of time it sees the wreath product $\mathbb{Z\wr\mathbb{Z}}/2\mathbb{Z}$.

\end{proof}

\bibliographystyle{plain}
\bibliography{concentrationref}

\end{document}